%% file: Bachmayr_Dahmen_Survey.tex
\documentclass{amsproc}
\usepackage[foot]{amsaddr}

\input{preamble-ams}

\usepackage{tikz}
\usetikzlibrary{arrows}
\usepackage{mathdots}
\usepackage{mathtools}

\newcommand{\pdom}{Y}

\newcommand{\sidx}{\mathcal{S}}
\newcommand{\pidx}{\mathcal{F}}

\newcommand{\rs}{{\rm x}}
\newcommand{\rp}{{\rm y}}

\newcommand{\cI}{\mathcal{I}}
\newcommand{\cH}{H}
\newcommand{\cF}{\mathcal{F}}

\def\Chi{\raise .3ex
\hbox{\large $\chi$}}

\title[Adaptive Low-Rank Approximations for Operator Equations]{Adaptive Low-Rank Approximations for Operator Equations: Accuracy Control and  Computational Complexity}

\author{Markus Bachmayr}
\address{Johannes Gutenberg-Universit\"at Mainz, Institut f\"ur Mathematik, Staudingerweg 9
55128 Mainz, Germany}
\curraddr{}
\email{bachmayr@uni-mainz.de}
\thanks{}

\author{Wolfgang Dahmen}
\address{University of South Carolina, Mathematics Department, LeConte College, 1523 Greene Street, Columbia, SC 29208, USA}
\curraddr{}
\email{wolfgang.anton.dahmen@googlemail.com}

\thanks{%
  This research was supported in part by
  the  NSF Grant    DMS 1720297,
  and by the SmartState and Williams-Hedberg Foundation.
  }
  

\date{\today}

\begin{document}

\maketitle

\begin{abstract}
The challenge of mastering computational tasks of enormous size tends to frequently override
questioning the quality of the numerical outcome in terms of {\em accuracy}.
By this we do not mean the accuracy within the discrete setting, which itself may  also be far from evident for ill-conditioned problems
or when iterative solvers are involved. By {\em accuracy-controlled computation} we mean the deviation of the 
numerical approximation from the exact solution of an underlying continuous problem in a relevant metric,
which has been the initiating interest in the first place. Can the accuracy of a numerical result be rigorously certified -- a question 
that is particularly important in the context of uncertainty quantification, when many possible sources of uncertainties interact.
This is the guiding question throughout this article, which reviews recent developments of low-rank approximation methods
for problems in high spatial dimensions. In particular, we highlight the role of {\em adaptivity} when dealing with such strongly
nonlinear methods that integrate in a natural way issues of  discrete and continuous accuracy.
\end{abstract}


\section{Introduction}

\subsection{Background}\label{ssec:1.1}
Quantifiable approximation, recovery, estimation of functions of a very large  and even infinite number of variables pose enormous challenges
in numerous application contexts of high current interest. 
 The discussion
in  this article is guided  by two sources of high (spatial) dimensionality, namely 

(I) partial differential equations (PDEs) in high dimensional phase space, and

(II) families of PDEs depending on a large number of {\em parameters} which could arise as design parameters or stem from parametrizing
random coefficient fields.

The electronic Schr\"odinger  equation for $N$ particles or
Fokker-Planck equations are typical representatives for (I). Both contain a second order diffusion operator as highest order term, which explains
the interest in considering {\em high dimensional} diffusion equations on  a product domain as a first model class to be studied. 
While from an analytical point of view this is, in principle, a very well understood problem, the focus here is on the spatial dimension  $d$ being large, even several hundreds or thousands.

Regarding (II), an intensely studied problem class concerns {\em parameter dependent} families of  operator equations %
\be
\label{modelpar}
 {R}(u;p)= 0
\quad \mbox{in}\quad\Omega,\quad u|_{\partial \Omega}=0,\quad p\in \mathcal{P},
\ee
where this time $\Omega\subset \R^m$ is a ``low-dimensional'' domain, typically with $m\in \{1,2,3\}$, but the operator depends 
on a parameter $p$ that may range over a high-dimensional (or even infinite-dimensional) parameter domain $\mathcal{P}$.
In an optimal control context, $\mathcal{P}$ can represent a set of design parameters. Another important instance of this type of problems arises when $p$ is actually a {\em random field} over some probability space modelling 
highly complex or micro-structured fields, such as permeability in a porous media flow. Expanding such a random field,
e.g., as a {\em Karhunen-Lo\`eve expansion}, one arrives at a representation of $p$ in terms of {\em parameter sequences}
$y= (y_1,y_2,\ldots) \in \pdom:= [-1,1]^\cI$ where in general $\cI=\N$.  Evaluating  $u(y)=u(p(y))$ for many parameter queries, computing
quantities of interest of the states $u(y)$, recovering such states from given observations, or estimating the underlying parameters are
typical tasks in the context of {\em uncertainty quantification}.

The common challenge in both problem scenarios (I) and (II) lies in recovering or approximating functions of a large or even infinite number of variables.
Classical numerical concepts based on (local) mesh refinements are of very limited use since they typically suffer from the {\em curse of dimensionality}, which roughly means
that numerical costs grow exponentially with the spatial dimension. The perhaps most promising remedy is to exploit some intrinsic
{\em sparsity} of solutions with respect to {\em a priori unknown} dictionaries or expansion systems. Technically this amounts 
to dealing with approximants that are parametrized in a typically rather nonlinear fashion, see Section \ref{ssec:nonlin}.
Such a strategy is indirectly taken up by the following quite common approach to such spatially high-dimensional problems that has been lately attracting significant attention.
One  starts from a (usually fixed) standard finite difference or finite 
element discretization, which initially gives rise to a discrete system of equations of enormous size. 
Numerical tractability is then achieved by solving the (fixed) discrete problem {\em approximately} restricting approximants
to a low-rank tensor format, see e.g. \cite{BNZ:14,KO:10,KressnerTobler:11} and the comments in Section \ref{ssec:nonlin}. The choice of the initial spatial resolution and the tensor ranks
is usually based on an educated guess with little or no provision in the solver to be updated according to some target 
tolerances.
So to speak as a tribute to the problem complexity, one is  tacitly content with controlling the accuracy 
of the {\em discrete approximate} solution with respect to the {\em  exact} solution of the {\em discrete problem} -- {\em discrete accuracy} -- but {\em not} with respect to the actual solution of 
the underlying continuous problem in a problem relevant metric. In stark contrast, in the sequel {\em accuracy} or {\em error} control will   always be understood 
in this latter sense with reference to the solution of the original {\em continuous} problem.    
In fact,  the predictive power of models like (I) or (II) 
depends among other things on the ability to quantify  this notion of accuracy which is the central theme in this article.
 
 Corresponding 
 concepts for low-rank or tensor approximations discussed in this article  target two
 central aims: (a) developing methods with a rigorous  quantification of  accuracy and (b) understanding
 how the entailed numerical complexity scales with increasing accuracy (with respect to the continuous solution). In fact, it is (b) that allows one
 to determine in the end whether a certain solver methodology is actually appropriate or whether numerical efficiency has  been
 compromised at the expense of any meaningful    accuracy.
 
It is important to note that a favorable answer to (b) is tantamount to saying that solutions to the continuous problem are ``well
 approximated by low-rank or tensor expressions''.  Such basic approximability properties are  
 discussed in Section \ref{sec:approximability} for both scenarios (I) and (II) to formulate benchmarks for the performance 
 of solvers.
 Devising numerical schemes that are able to best exploit these approximability properties  
 requires a proper balancing of both
 error sources, namely keeping ranks finite and discretizing corresponding tensor factors. Ideally, ranks and low-dimensional discretizations should evolve in a completely intertwined fashion,
 which is a highly nonlinear process. Insisting on  {\em error controlled}  realization of such processes is the overarching objective of subsequent discussions. A central message is that
 this is only possible by respecting and exploiting characteristics of the continuous problem, such as intrinsic metrics, which strongly link
 the discrete and continuous setting.  The standard paradigm ``first discretize, then analyze''  is thus turned around
 in that computation is ``pretended'' to take place in the {\em infinite} dimensional context. In fact, a convergent iteration formulated in the infinite dimensional setting is shown to remain convergent when executed only approximately within suitable accuracy tolerances.
 Discretizations are therefore never fixed beforehand, but adapt at any given stage of the ``outer iteration''  to such dynamic tolerances.

Therefore, it is important to note that {\em a priori estimates} alone are not quite sufficient 
  since they often involve unknown quantities or are valid under assumptions that may not hold or are
 impossible to check. Instead, computable rigorous  {\em a posteriori} error control is a key element of the proposed approach which must
 exploit the structure of the continuous problem.
This can be carried out most conveniently for those instances of (I) and (II) where, for a proper choice of a Hilbert space $V$,
 corresponding {\em weak formulations}
\be
\label{weak}
a(u,v) =f(v),\quad v\in V,
\ee
are well-posed. This means
  that the bilinear form $a(\cdot,\cdot):V\times V\to \R$ is  symmetric, bounded and coercive, i.e., \eqref{weak} is
{\em $V$-elliptic}. $a(\cdot,\cdot)$   therefore defines an equivalent scalar product
on $V$, which means   that the operator $A:V\to V'$, defined by $(A w)(v)= a(w,v)$, $w,v\in V$, is boundedly invertible. 

One reason for discussing scenario (I) in comparison to (II) is that the respective {\em energy spaces} $V$ differ 
in an essential way, which will be seen to affect the numerical methods in an equally essential way. Nevertheless, the common ground
for both (I) and (II) is the ellipticity of $a(\cdot,\cdot)$ in \eqref{weak}, or equivalently, that the induced mapping $A:V\to V'$ is an isomorphism.
This means
 there exist constants $0<c_a\le C_a<\infty$ such that 
\be
\label{ellipA}
c_a\|w\|_V \le \|A w\|_{V'}\le C_a\|w\|_V,\quad w\in V.
\ee
Since $A(u-w)= f- Aw$ this implies in particular, that
\be
\label{res}
c_a\|u-w\|_V \le \|f- Aw\|_{V'} \le C_a\|u-w\|_V,\quad w\in V,
\ee
i.e., errors in the trial norm are equivalent to the dual norm of corresponding residuals. It is to be stressed that ``residual'' stands here
for the defect in the {\em infinite-dimensional} setting not within a fixed finite-dimensional discretization.

Exploiting ellipticity through such error-residual relations is a corner stone of the proposed approach.
 Of course, this is in principle also the starting point of adaptive finite element techniques in low dimensions. One then uses duality arguments
to derive sharp computatble approximations to $ \|f- Aw\|_{V'}$. Since these techniques rely crucially on {\em localization principles}
which are natural in a finite element framework, they are,  being fully subject
to the curse of dimensionality, in this form {\em infeasible} in high-dimensional regimes. In particular, any type of mesh would be meaningless in scenario (II) where the number of 
relevant ``activated'' parameters may depend on the target accuracy and not be known beforehand.

Therefore, we will exploit \eqref{res} in a different way  inspired by
adaptive wavelet methods \cite{Cohen:01,Cohen:02}. It hinges on first identifying a {\em Riesz basis} for the energy space $V$.
This allows one to transform \eqref{weak} into an {\em equivalent infinite-dimensional} system of linear equations where
the matrix representation $\bA$ of the operator $A$ is now an {\em isomorphism} from the space of square-summable sequences onto \emph{itself},
rather than mapping $V$ onto a {\em different,} less smooth space $V'$.   Thus, errors are measured now in the same (Euclidean) norm
as residuals.  This transformation 
``preconditions'' the problem already on the infinite-dimensional level.  An error-controlled approximation of residuals then reduces primarily 
to an adaptive error-controlled approximate  application of the matrix $\bA$ within a given low-rank or tensor format, fully intertwining
low-rank approximability and spatial sparsity of low-dimensional tensor factors.
 
 Corresponding computational realizations build essentially on recent important developments of tensor calculus, especially,
 for hierarchical tensor formats \cite{BSU:18,Ballani:13,Falco:10,Kolda:09,Grasedyck:13,Grasedyck:10,Hackbusch:09-1,Hackbusch:12,Oseledets:09,Oseledets:11}.
  As previously indicated, a price for rendering high-dimensional problems practically tractable is to employ non-standard parametrization
 formats for approximants which naturally complicates numerical processing. The following section attempts to put this into
 a perspective which is relevant for the remainder of the discussion.

\subsection{Nonlinear Approximation and Parametrization Formats}\label{ssec:nonlin}

A numerical approximation of a function $u \colon \Omega \to \R$, with $\Omega \subset \R^d$, can be regarded as an algorithmic template for a computable substitute that can be parametrized to resemble $u$. In the most classical scenario $u$ is supposed to be an element of a Banach space $X$ endowed with a norm $\|\cdot\|_X$, and one looks for increasingly better approximations to $u$ from a {\em preselected} sequence
$X_N\subset X$ of $N$-dimensional linear subspaces spanned by computationally accessible basis functions $\{\varphi_1,\varphi_2,\ldots,
\varphi_N\}$
such as polynomials, splines, or finite elements. Thus, 
 for each $N\in\N$, one seeks a computational prescription 
\[ 
\Phi_{\rm lin}(x;N, \uc_1, \ldots,\uc_N) := \sum_{i=1}^N \uc_i \varphi_i(x)
\]
 parametrized by the coefficient vector $\mathbf{u}=(\uc_1,\ldots,\uc_N)^T$   which needs to be adjusted to yield small  $\|\Phi_{\rm lin}(\cdot;N, \bu) - u\|_X$. This is also termed \emph{linear} approximation, since the approximant is sought in the linear subspace spanned by 
 a {\em preselected} set of   $N$ basis functions.

 When the spaces $X_n$ are nested $X_n\subset X_{n+1}$ it is possible to construct an {\em infinite collection} $\{\varphi_1,\varphi_2,\ldots\}$
 such that each $X_N$ is spanned by the first $N$ elements of this collection  and the whole collection forms a {\em basis} for $X$ in an appropriate
 sense, depending on the nature of the space $X$. For instance, {\em hierarchical} or {\em wavelet bases} fall into this category covering, in particular, the finite element setting.
 
 This latter point of view is useful as it offers a convenient unifying framework for {\em adapting} an approximation procedure to each specific instance of approximants.
With additional parameters $\lambda_1,\ldots,\lambda_N \in \N$, one may now consider algorithmic templates of the form
\[ 
\Phi_{\rm nonlin}(x; N, \uc_1, \ldots,\uc_N, \lambda_1,\ldots,\lambda_N):=\sum_{i=1}^N \uc_i \varphi_{\lambda_i}(x),
\] 
that allow one to pick  those basis functions that are best suited to approximate the target $u$ within a given budget $N$. Now approximations
are generated from the {\em nonlinear sets}
\begin{equation}\label{eq:sigmadef}
\Sigma_N := \bigcup_{\Lambda \subset \N, \#\Lambda = N} \Bigl\{\sum_{\lambda\in\Lambda} \bu_\lambda\varphi_\lambda: \bu_\Lambda\in\R^N \Bigr\}
\subset X,\qquad \bu_\Lambda :=(\bu_\lambda)_{\lambda\in\Lambda}.
\end{equation}
Since this requires choosing or {\em activating} $N$ among the infinitely many basis functions, such a process is called {\em nonlinear approximation}.
In practical realizations one typically generates an increasing sequence of activated index sets $\Lambda_n$, where the choice of $\Lambda_{n+1}$
exploits information gained from the preceding stage represented by $\Lambda_n$. This form of nonlinear approximation is referred to as {\em adaptive approximation}.  Its rigorous foundation very much relies on being able to quantify the accuracy obtained at a previous stage $n$. It is tantamount to asking
for {\em certified a posteriori error bounds}, which is a recurrent theme throughout this article. This issue is precisely what requires
intertwining the discrete and continuous setting and exploiting, in particular, intrinsic problem metrics.

These concepts have been quite successful and are by now fairly well understood for low spatial dimensions, typically $d\le 3$. In particular, the performance
of such schemes is essentially governed by {\em Besov regularity}, that is, smoothness in $L_p$ spaces where $p$ is allowed to be less than one,
see \cite{DeV:acta}.  In correspondence to the isotropic nature of such classical regularity notions, improved accuracy is achieved by
increasing {\em spatial localization}. Isotropic localization, however, is precisely what causes the curse of dimensionality.
  
In cases with moderately large $d$ (depending on the problem, for instance, the low two-digit regime), a successful remedy is based on employing 
{\em product-type} basis functions $\varphi_\lambda(x) = \varphi_{\lambda_1}^{(1)}(x_1)\cdots \varphi^{(d)}_{\lambda_d} (x_d)$, where
$x= (x_1,\ldots,x_d)$, each $x_i$ belongs to $\R^m$ for small $m$, and the $\varphi^{(i)}_j$ form a low-dimensional basis of the type 
discussed above. Associating with each index $\lambda_i$ a {\em scale} denoted by $|\lambda_i|$, typically indicating that 
$\operatorname{diam}(\operatorname{supp}\varphi^{(i)}_{\lambda_i})\sim 2^{-|\lambda_i|}$, one considers, for instance, expansions of the form
$$
\sum_{|\lambda_1|+\cdots+|\lambda_d|\le L} \uc_\lambda \varphi_\lambda(x).
$$
In other words, the a priori activated summands never involve simultaneously many fine scale basis functions, and the complexity of such expressions
scales like $2^L L^{d-1}$. This is the concept of {\em sparse grid} or {\em hyperbolic cross approximation}. One thus gives up
on isotropic localization. For such approximations to provide high accuracy one has to demand, however, correspondingly high
regularity, which in this context means controlled {\em mixed partial derivatives} of appropriate order. While in this form such schemes are still linear, they lend themselves to nonlinear versions in natural ways, see \cite{SSprod:08}. 

However, the principle of adaptively activating basis functions from any {\em preselected} basis or dictionary may no longer suffice for large $d$
to warrant affordable computational cost in an accuracy controlled approximation. Instead one has to resort to yet stronger notions of nonlinearity that may allow one to
better capture a hidden sparsity related to the problem at hand, inadvertedly asking for a deeper understanding of the underlying
continuous model. In other words, the actual  ``building blocks'' for representing the target function in the spirit of classical harmonic analysis, 
should {\em no longer be pre-determined} but rather depend on the problem at hand.

A particularly flexible parametrization format that has been lately attracting considerable attention are {\em deep neural networks} (DNN)
of the form 
\be
\label{NN}
\Phi_{\rm nn}\bigl( x ; \;(\bv_{\ell, i})  \bigr) := \sigma_1 \biggl( \sum_{i_1} \bv_{1,i_1} \sigma_{2, i_1} \Bigl(\cdots \sigma_{L, i_{L-1} }\Bigl(  \sum_{i_L} \bv_{L, i_L} x_{i_L}   \Bigr) \cdots\Bigr) \biggr).
\ee
That is, the mapping taking possibly high dimensional inputs $x$ into an approximation to $u(x)$ is a concatenation of affine maps followed by 
a componentwise nonlinear map, called ``activation function''. This approach is based on the presumption that in many application scenarios target objects are 
close to what can be covered by such parametrizations.
Although existence of efficient approximations of this type has been shown for various cases of interest \cite{Poggio:17,Yarotzki:17}, guaranteeing convergence and error control in the actual computation of such approximations currently remains a wide open problem. 

The type of approximation that we focus on here is somewhat ``less nonlinear'' than \eqref{NN} but of sufficiently strong nonlinearity to make it suitable for a range of problems with very large $d$. The principle is easiest to explain for $d=2$: assume that we have product basis functions $\varphi_{i,j}(x) = \varphi^{(1)}_{i}(x_1)\,\varphi^{(2)}_{j}(x_2)$, $i,j \in \N$. We then consider adaptive low-rank approximations of the form
\begin{equation}
\label{abstract lowrank d=2}
\begin{aligned}
&  \Phi_{\rm lr}\Bigl(x; r, N_1, N_2, \substack{ \uc^1_{1,1}, \ldots, \uc^1_{r,N_1},\\  \lambda_1,\ldots,\lambda_{N_1},} \;\substack{ \uc^2_{1,1}, \ldots, \uc^2_{r,N_2}, \\ \mu_1,\ldots,\mu_{N_2}}   \Bigr) 
   := \sum_{k=1}^r \sum_{i = 1}^{N_1} \sum_{j = 1}^{N_2} \uc^1_{k, i} \uc^2_{k, j} \varphi^{(1)}_{\lambda_{i}}(x_1) \varphi^{(2)}_{\mu_{j}}(x_2) \\
   &\qquad\qquad\qquad = \sum_{k=1}^r \biggl(\sum_{i= 1}^{N_1} \uc^1_{k, i} \varphi^{(1)}_{\lambda_{i}}(x_1)  \biggr) \biggl( \sum_{j = 1}^{N_2} 
  \uc^2_{k, j} \varphi^{(2)}_{\mu_{j}}(x_2) \biggr)
  =: \sum_{k=1}^r \uc^1_k(x_1)\, \uc^2_k(x_2).
  \end{aligned}
\end{equation}
with the additional \emph{rank parameter} $r$. 

To be specific, suppose we seek to approximate elements $u$ in a tensor product 
Hilbert space $\cH=\cH_1\otimes\cH_2$ where $\cH_i$ are separable Hilbert spaces spanned by the orthonormal bases
$\{\varphi^{(q)}_i\}_{i\in\N}$, $q=1,2$, respectively. Ideally, given a target tolerance $\e>0$, one would like to adapt
the rank $r$ as well as the spatial resolution of ``best-suited'' modes $u_k^1, u_k^2$ as functions of $x_1, x_2$, respectively,
which is obviously
a highly nonlinear problem. ``Best-suited'' means that a given target accuracy can be met with rank $r$ as small as possible.
  In the present two-dimensional case, this amounts to approximating the  {\em infinite} matrix 
  $\bU$ of basis coefficients with respect to the full tensor product basis $\{\varphi^{(1)}_i \otimes \varphi^{(2)}_j\}_{(i,j)\in \N^2}$
   with respect to the Frobenius  (or Hilbert-Schmidt) norm within the same
  target tolerance,
  by a finite {\em low-rank approximation}  $\bU^1 (\bU^2)^\top$, 
  where $\bU^1 = (u^1_{i,k})_{i,k=1}^{N_1,r}$, $\bU^2 = (u^2_{i,k})_{i,k=1}^{N_2,r}$.  

  Denoting by  $\bU_{N_1,N_2}$  the finite section of $\bU$ corresponding to the truncated basis 
 $\{\varphi^{(1)}_i \otimes \varphi^{(2)}_j\}_{i\le N_1,j\le N_2}$, 
  the error $\|\bU^1 (\bU^2)^\top -\bU\|_{\rm HS}$
  can be split into two portions $E_1:=\|\bU^1 (\bU^2)^\top -\bU_{N_1,N_2}\|_{\rm HS}$ and $E_2:=\| \bU_{N_1,N_2}-\bU  \|_{\rm HS}$.
  The vast majority of studies  focuses on
 controlling only  $E_1$, essentially ignoring how the overall error depends on the spatial discretizations represented by 
 $N_1, N_2$, determining $E_2$. In the current situation $d=2$, given a target accuracy $\e_1$ for $E_1$, the minimal rank $r$ can then be determined  by means of the   {\em Singular Value Decomposition} (SVD) of $\bU_{N_1,N_2}$.
  Observing that in many cases one can achieve $E_1\le \e_1$ with $r \ll N_1, N_2$ indeed  signals
  a substantial complexity reduction in the number of required coefficients $(N_1 + N_2) r$ compared to a full array representation involving $N_1 N_2$
  terms.  However, this by itself does not say much about the computational cost of approximately solving the original
  problem within some target tolerance $\e$, unless the error portions $E_1$ and $E_2$ are essentially {\em balanced} which
  actually depends  on the (full) {\em Hilbert-Schmidt decomposition} (HSD) of $\bU$. 
  Such an  assessment can, however, not be reached
  from a linear algebra perspective alone. The central objective of this article is to highlight concepts that allow one to certifiably control the total error $ \|\bU^1 (\bU^2)^\top -\bU\|_{\rm HS}$ for large $d$, which must involve the underlying continuous model.

 \subsection{Layout}
 Unfortunately, a straightforward extension of the format \eqref{abstract lowrank d=2} to $d>2$ lacks in general stability. Section \ref{sec:tensors}
 is therefore devoted to a brief discussion of alternate stable tensor formats and their main properties.
  In Section \ref{sec:solstrat} we outline a  solution strategy for a general class of high-dimensional {\em elliptic} operator equations covering the 
 two scenarios (I) and (II), addressed above, as special cases. 
 Section \ref{sec:approximability} addresses the 
 approximability of solutions to high-dimensional elliptic problems by tensor methods. 
 Again, the findings for the two   scenarios (I) and (II) turn out to be quite different.
 Finally, in Section \ref{sec:complexity} we present techniques for controlling the computational complexity of approximations as well as of numerical schemes based on the general strategy from Section \ref{sec:solstrat}, and discuss the conclusions for scenarios (I) and (II).

\section{Subspace-Based Tensor Formats and Nonlinear Approximation}\label{sec:tensors}
\subsection{Tensor Formats}

This
section offers a brief review of   approximating   elements $u$ in a tensor product 
\be
\label{cH}
\cH:=\cH_1\otimes\cdots\otimes\cH_d,
\ee
 of separable Hilbert spaces $\cH_i$, endowed with the unique {\em cross-norm} $\|\cdot\|_\cH$ 
associated with an inner product $\langle \cdot ,\cdot \rangle_\cH$ satisfying %
$\langle \bigotimes_{i=1}^d v_i,\bigotimes_{i=1}^d w_i \rangle_\cH = \prod_{i=1}^d\langle v_i,w_i \rangle_{\cH_i}$. 

 A natural extension of \eqref{abstract lowrank d=2} to spatial dimensions $d\ge 3$ would be  to seek approximations in the form
\be
\label{CP1}
u_r(x) = \sum_{k=1}^r u^1_{k}(x_1) \, u^2_{k}(x_2)\cdots u^d_{k}(x_d),
\ee
which is often referred to as the {\em canonical} or CP-{\em format}, \cite{Beylkin:02,Hackbusch:12}. Again, the issue is to find suitable modes $u_k^i$  that 
warrant high accuracy at the expense of a low rank $r$. However,  first, the class of rank-$r$ tensors is not closed, and second, best approximations from such classes do in general not exist \cite{Silva:08}. Therefore, representations in this format   need to be treated with care and
the observed  inherent instability   suggests looking for more stable tensor fomats \cite{Hackbusch:09-1,Hackbusch:12,Kolda:09,Lathauwer:00,Oseledets:09,Tucker:64}.

The canonical format is indeed only one special instance of possible tensor representation formats.
A particular type of format with certain restrictions on the components are {\em tensor networks};
for a comprehensive discussion we refer, for instance, to \cite{BSU:18}. In what follows we confine the discussion
to the important subclass of {\em tree networks} that heavily draws on   the favorable features of the case $d=2$.

\subsection{Tree-Based Hierarchical Formats}\label{ssec:tree}

To benefit from SVD concepts also when 
 $d>2$, the key idea is to view a (discrete) tensor $\bu$ of order $d$, say, as a {\em matrix} by viewing all multi-indices with respect to some subset
 $\alpha\subset \alpha^*:=\{1,\ldots,d\}$ as row-indices while the multi-indices with respect to the remaining variables in $\alpha^c := \alpha^*\setminus \alpha$ form the column indices. An SVD of this matrix gives rise to singular left and right vectors which are now tensors of 
 order $\#\alpha, \#\alpha^c$, respectively. By successively further decomposing the left singular vectors then leads to approximations
 in terms of tensors of lower and lower order. An underlying successive splitting of groups of variables into smaller ones can be 
 conveniently organized by a so called (binary) {\em dimension tree}
 $\mathbb{T}_d$ whose nodes $\alpha$ are subsets of the {\em root} $ \alpha^*=
 \{1,\ldots,d\}$. $\mathbb{T}_d$ always contains the root $\alpha^*$. Each node $\alpha$ of cardinality $\#\alpha > 1$ has  a unique pair of children $\alpha_1,\alpha_2\in \mathbb{T}_d$ such that $\alpha=\alpha_1\cup\alpha_2$, $\alpha_1\cap\alpha_2=\emptyset$, $\#\alpha_i <\#\alpha$.
 The set $L(\mathbb{T}_d)$ of leaves of $\mathbb{T}_d$ is comprised of the nodes  $\alpha$ of cardinality one.   
   The corresponding so-called {\em hierarchical tensor format} was developed by Hackbusch and K\"uhn \cite{Hackbusch:09-1}, see also 
 \cite{Oseledets:09}. 
 
 It will be important for what follows to formalize these concepts in the context of {\em infinite-dimensional} tensor product Hilbert spaces
 $\cH$ of the form \eqref{cH}.
  To that end, consider  again for any $\alpha\in \T_d$ 
the grouping $\cH = \cH_\alpha\otimes \cH_{\alpha^c}$. 
The continuous and linear extension of $M_\alpha(u_\alpha\otimes u_{\alpha^c})v := u_\alpha\langle u_{\alpha^c},v\rangle_{H_{\alpha^c}}$, $v\in H_{\alpha^c}$, to an operator $M_\alpha(u)$ for $u\in H$, 
\be
\label{HS}
  M_\alpha(u) : \cH_{\alpha^c}\to \cH_\alpha %
\eeqn
is a Hilbert-Schmidt operator from $\cH_{\alpha^c}$ to $\cH_\alpha$, i.e.,
$$
\|M_\alpha(u)\|_{\rm HS} = \|u\|_\cH,
$$
 and hence is compact. In a slight abuse of terminology we refer to $M_\alpha(u)$ as a {\em matricization} of $u$.
 By the spectral theorem for compact operators, there exist orthonormal systems $\{u^\alpha_k\}_{k\in \N}$, $\{u^{\alpha^c}_k\}_{k\in \N}$ (depending on $u$)
 of $\cH_\alpha,\cH_{\alpha^c}$, respectively, as well as a sequence of nonnegative numbers $\sigma^\alpha_k$ tending to zero such that
 \be
 \label{HSrep}
 u = \sum_{k=1}^\infty \sigma^\alpha_k u^\alpha_k\otimes u^{\alpha^c}_k.
 \ee
 The $\sigma^\alpha_k$ are the singular values associated with the matricization $M_\alpha(u)$.
 If $\sigma^\alpha_k=0$ for $k>r_\alpha$ we say that $u$ has $\alpha$-rank (at most) $r_\alpha=r_\alpha(u)\in \N_0\cup\{\infty\}$ which is the dimension of the range 
 \be
 \label{Ualpha}
 \cU_\alpha = \cU_\alpha(u) := \overline{{\rm range}\,(M_\alpha(u))}
 \ee
 of $M_\alpha(u)$ and hence equals the $\alpha^c$-rank of $u$. 
 In general,
  the subspaces $\cU_{\alpha,r} := {\rm span}\,\{u^\alpha_k: k=1,\ldots,r\}$ are {\em optimal} in the sense that
 \be
 \label{best}
\Big \|u- \sum_{k=1}^r \langle u,u^\alpha_k\otimes u^{\alpha^c}_k\rangle_\cH u^\alpha_k\otimes u^{\alpha^c}_k\Big\|_\cH =
 \min\,\big\{ \|u - w\|_\cH: w \,\text{ has $\alpha$-rank }\le r\} . %
 \ee
If one further decomposes $\alpha = \alpha_1\cup\alpha_2$ into its children $\alpha_1,\alpha_2$, according to $\T_d$, one has
 the {\em nestedness property}
 \be
  \label{nesting}
 \cU_\alpha \subseteq  \cU_{\alpha_1}\otimes \cU_{\alpha_2},
 \ee 
 see \cite{Hackbusch:12} (cf.~Corollary 6.18 and
Theorem 6.31 there).
  Whenever \eqref{nesting} holds one can recursively decompose $u\in \cU_\alpha$ as follows. For any orthonormal bases $\{u^\beta_k\}_{k=1}^{r_\beta}$ of $\cU_\beta$, $\beta\in \T_d$, one can write
   \be
\label{recursion}
u^\alpha_k = \sum_{k_1=1}^{r_{\alpha_1}}\sum_{k_2=1}^{r_{\alpha_2}}\bB^\alpha(k_1,k_2,k)u^{\alpha_1}_{k_1}\otimes u^{\alpha_2}_{k_2}.
\quad \alpha \in \mathbb{T}_d \setminus L(\mathbb{T}_d).
\ee
where $\bB^\alpha(k_1,k_2,k) = \big\langle u^\alpha_k, u^{\alpha_1}_{k_1}\otimes u^{\alpha_2}_{k_2}\big\rangle$.
The tensors $\bB^\alpha$ are referred to as {\em component}  or {\em transfer tensors}. Hence  any $u\in \cH$ can be recovered exactly in terms of its expansion tensor $u$ as
\be
\label{recoveru}
u =  \sum_{k_1=1}^{r_{\alpha^*_1}}\sum_{k_2=1}^{r_{\alpha^*_2}}\bB^{\alpha^*}(k_1,k_2)u^{\alpha^*_1}_{k_1}\otimes u^{\alpha^*_2}_{k_2},
\ee
with $r_{\alpha^*_1}=r_{\alpha^*_2}\in \N\cup\{\infty\}$.
A recursive substitution of \eqref{recursion} therefore parametrizes a $u\in \cU_{\alpha^*_1}\otimes
\cU_{\alpha^*_2}$ in terms of the transfer tensors $\bB^\alpha$, $\alpha\in \T\setminus L(\T)$ and the {\em  mode frames} 
$u^\mu\in \cH_\mu$, $\mu\in \{1,\ldots,d\}$. Using the SVD to successively generate the orthonormal systems $(u^\alpha_k)_{k=1}^{r_\alpha}$,
is referred to as {\em hierarchical singular value decomposition} $\cH$SVD.

 Before proceeding let us emphasize that such hierarchical decompositions are actually identified by the set of {\em pairs} 
 \be
 \label{E}
 \mathbb{E} := \big\{e=\{\alpha,\alpha^c\}:  \alpha
 \in \T_d\setminus \{\alpha^*\}\big\},
 \ee
 called the set of {\em effective edges}, see \cite{BSU:18}. Different dimension trees can give rise to the same $\mathbb{E}$ and hence to the same matricizations.
 Such trees are in that sense equivalent. 

 In other words, for $\{\alpha,\alpha^c\}\in \mathbb{E}$ and any $\T_d$ in the equivalence class determined by
 $\mathbb{E}$, either $\alpha$ or $\alpha^c$  belongs to $\T_d$, and this element is the representer of $e=\{\alpha,\alpha^c\}$
 denoted by $[e]\in \T_d$. 
 It is easy to see that
  \be
 \label{card}
  \#\mathbb{E} = 2d-3 =: E.
 \ee
We fix in what follows an enumeration $\{e_i\}_{i=1}^E$ of the effective edges in $\mathbb{E}$.
 
 As a consequence, for $u\in \cH$ and a given $\mathbb{E}$ (and hence all dimension trees in the corresponding equivalence class) we can write for the  corresponding matricizations and  subspaces $M_i(u):= M_{[e_i]}(u)$, $\cU_i=\cU_i(u)=\cU_{[e_i]}(u)$, respectively.
 One can then associate with $u$ its $\mathbb{E}$-{\em rank}
 \be
 \label{Erank}
 \rr{r}_{\mathbb{E}}(u) \in (\N_0\cup \{\infty\})^E := \big( r_i\big)_{i=1}^E,
 \ee
where $r_i = r_i(u) := {\rm dim}\,\cU_i(u)= \rank(M_i(u))$,  $ i=1,\ldots, E$.
 
In summary, any $u\in \cH$ can be parametrized by 
\be
\label{param}
\mathfrak{p}(u):= \big(\rr{r}_{\mathbb{E}}(u), (\bB^{\alpha})_{\alpha \in \mathbb{T}_d\setminus L(\mathbb{T}_d)}, (u^\mu)_{\mu=1}^d\big).
\ee
Such a 
parametrization is obviously {\em finite} if all transfer tensors have finite ranks, in particular, when all entries in $\rr{r}_{\mathbb{E}}$
are finite. Moreover, the leaf elements $u^\mu$ must admit a finite parametrization, e.g. in terms of a truncated orthonormal basis 
for $\cH_\mu$. Thus $\#\mathfrak{p}(u)\in \N\cup\{\infty\}$ is the total number of parameters needed to represent $u$ in the above 
hierarchical format. 

The appeal of such hierarchical tensor decompositions for high-dimensional approximation  lies in the following facts, see also 
\cite{Grasedyck:10,Grasedyck:13}.
\begin{remark}
\label{rem:compl}
Suppose for a moment that the maximal number of parameters to determine each leaf modes $u^\mu$ is $n$
and that all entries of $\rr{r}_{\mathbb{E}}(u)$ are bounded by a fixed $r\in \N$. In view of \eqref{card}, 
we see that the number of parameters to be stored is of the order
$O(dr^3 +ndr)$. The total   numerical complexity of computing a $\Hcal$SVD under these premises is of the
order of $O(dr^4 +dr^2n)$. In such a situation the {\em representation and computational complexities} of $u$ depend only {\em linearly}
 on the spatial dimension $d$.
\end{remark}
 Thus, approximating a given $u\in \cH$ by properly truncated versions $u_\e$ for which 
$\#\mathfrak{p}(u_\e)$ is of moderate size may open, at least for certain classes of target functions $u\in \cH$, a promising avenue to
mitigate the curse of dimensionality. 

To make this somewhat more precise, note first that the entries of the rank vectors $\rr{r}_{\mathbb{E}}(u)$ must satisfy 
in the finite case some
{\em compatibility conditions} in order to comply with \eqref{nesting}. 
 In fact, it   follows from the results just stated
that for $i=1,\ldots,E$   one has $r_i(u)\leq r_{[e_i]_1}(u)  r_{[e_i]_2}(u)$.
  For necessary and sufficient conditions on a rank vector $\rr{r}_{\mathbb{E}}$
 we refer to \cite[Section 11.2.3]{Hackbusch:12}. In what follows we denote by 
\begin{equation}
\label{def:R-H}
\mfR=\mfR_{\mathbb{E}}\subset (\N_0 \cup \{\infty\})^{ \mathbb{E}}
\end{equation}
 the set of all hierarchical rank vectors satisfying the compatibility conditions for nestedness \eqref{nesting}.
 For any $\rr{r}\in \mfR_{\mathbb{E}}$ we define then
\begin{equation}
\label{eq:hier_tensorset}
  \Hcal(\rr{r}) := \bigl\{ {u} \in
 \cH \colon  r_i(u) \leq {r}_i \text{
  for all $i=1,\ldots,E$} \bigr\} \,.
\end{equation}
In what follows we will be concerned with employing elements from such hierarchical tensor classes to approximate solutions
to problems of type (I) and (II).

In the previous discussion the dimension tree $\T_d$, and hence $\mathbb{E}$, were kept fixed.
It is certainly an interesting question of how to adapt 
$\T_d$ to a given approximand $u\in \cH$, so as to warrant good approximations at the expense of possibly small ranks $\rr{r}\in \mfR_{\mathbb{E}}$.  We will not address this issue in any depth but pause to
 briefly mention
the following cases of interest. Balanced trees arise when successively splitting the nodes $\alpha$ into two children
of roughly the same cardinality.   The opposite case of a linear tree amounts to always choosing the left child $\alpha_1$ have cardinality one (i.e., as a leaf node), while collecting the other entries of $\alpha$ in the right child $\alpha_2$.
The resulting format is termed {\em Tensor Train} (TT) format \cite{Oseledets:09}. 
Combining pairs of mode frames and transfer tensors, this gives rise to an explicit  (entry-wise) multilinear representation
\be
\label{TT}
u(x_1,\ldots,x_d) = \sum_{k_1=1}^{r_1}\cdots\sum_{k_{d-1}=1}^{r_{d-1}}u^1(x_1,k_1)\,u^2(k_1,x_2,k_2)\,u^3(k_2,x_3,k_3)\cdots u^d_{k_{d-1}}(x_d)
\ee
with remaining rank parameters $r_1,r_2,\ldots,r_{d-1}$.
This format is also known as {\em matrix product states} in physics. The representation complexity is
now easily seen to be $O(dnr^2)$ which again nourishes hope to defeat the curse of dimensionality for suitably nearly-sparse $u\in \cH$,
see e.g. \cite{Oseledets:09,BSU:18} for more details.

\subsection{Complexity Reduction}\label{ssec:properties}
In contrast to the canonical format the above subspace based tensor format warrants now the {\em existence} of best approximations in $\cH(\rr{r})$.
Following \cite{Falco:10,Hackbusch:12}, we can  can state this as follows.
\begin{theorem}
\label{thm:htucker_bestapprox}
For $u\in \cH$ and $\T_d$   a dimension tree from the equivalence class of $\mathbb{E}$,
 let $\rr{r} \in \mfR_{\mathbb{E}}$ with $0 \leq r_i \leq
r_i( {u})$ for $i=1,\ldots,E$. Then there exists $ {v}
\in \Hcal(\rr{r})$ such that
\begin{equation*}
  \norm{ {u} -  {v}}_{\cH} =
  \min
  \bigl\{ \norm{ {u} -  {w}}_{\cH} \colon
   r_i( {w})\leq r_i, \, i=1,\ldots,E \bigr\} \,.
\end{equation*}
\end{theorem}

Note that as a consequence of \eqref{nesting}, an element providing a best approximation to $u$ as in Theorem \ref{thm:htucker_bestapprox} can be written as the result of a projection applied to $u$; we make use of this fact in Section \ref{sec:complexity}.
While the computation of best approximations is usually
not computationally feasible, {\em near-best approximations} can be obtained at affordable cost: The $\cH$SVD produces orthonormal bases $(u^i_k)_{k=1}^{r_i(u)}$ for each $\cU_i$, associated to the decreasing sequences $(\sigma_k^{(i)})_{k=1}^{r_i}$ of singular values from \eqref{HSrep}.  
Truncation to lower ranks $\tilde r_i \leq r_i(u)$ for each $i$ gives rise to a projection $\mfH_{\tilde{\rr{r}}}$ into $\cH(\tilde{\rr{r}})$. It amounts to truncating the corresponding transfer tensors and mode frames in \eqref{param}, leading to approximations with errors bounded in terms of the quantities
  \begin{equation}
\label{eq:htucker_err}
 \tau_{i,r}(u):= \Big(\sum_{k>r}|\sigma^{(i)}_k(u)|^2\Big)^{1/2},\qquad
t_\rr{\tilde r}(u) %
   := \Bigl( \sum_{i=1}^E  \tau_{i, \tilde r_i}(u)^2
  \Bigr)^{\frac12} \,,
\end{equation}
with the following quasi-optimality property shown in \cite{Grasedyck:10}.

\begin{theorem}
\label{thm:hier_tensor_recompression_est}
Let $ {u} \in \cH$.
{Then for hierarchical ranks $\rr{\tilde r} = (\tilde
r_i)_{i=1}^E\in \mfR_\Hcal$,} we have{
\begin{equation*}
 \norm{ {u} - \mfH_{\tilde{\rr{r}}}(u)}_{\cH} \leq t_\rr{\tilde r}(u)
  \leq \sqrt{2 d - 3} \, \inf
 \bigl\{\norm{ {u} -  {v}}_{\cH} \colon  {v} \in
 \Hcal(\rr{\tilde r}) \bigr\} \,.
\end{equation*}}
\end{theorem}

The operator $\mfH_{\tilde{\rr{r}}}$ in essence  truncates the $i$-ranks or equivalently sets corresponding $i$-singular values to zero
and is therefore sometimes  referred to
as {\em hard thresholding}.
While $\mfH_{\rr{r}}$  does not provide truly best rank $\rr{r}$-approximations, the inflating constant depends only
mildly on the dimension $d$ and it is computationally feasible at affordable cost. For detailed discussions of how to realize $\mfH_{\tilde{\rr{r}}}$ efficiently we refer to \cite{Grasedyck:10}.

 Since $\Hcal(\rr{r})$ is %
 not a linear space, arithmetic calculations with hierarchical tensors will inevitably increase ranks. To control complexity it is important to approximate
 an element in $\cH(\rr{r})$ by one of smaller ranks  in $\cH(\tilde{\rr{r}})$, $\tilde{\rr{r}}\le \rr{r}$, as well as possible.
 This is often referred to as {\em recompression}. We defer the discussion of principles of how to properly balance
 a rank reduction by recompression with the entailed loss of accuracy to a later section. 

An alternative type of rank reduction is based on the concept of {\em soft thresholding} of singular values. It has the advantage of preserving the contraction properties of iterative schemes, as considered in more detail in Section \ref{sssec:st}. Soft thresholding as a scalar operation is defined for a given thresholding parameter $\eta>0$ as
\be
\label{softthres}
s_\eta(x) := {\rm sgn}(x)\max\{|x|-\eta,0\},\quad x\in \R.
\ee
The crucial property of this operation is its {\em non-expansiveness}, that is,
\be
\label{nonexp}
|s_\eta(x)- s_\eta(y)|\le |x-y|,\quad x,y\in \R.
\ee
Note that $s_\eta$ can be characterized variationally by
\be
\label{var}
s_\eta(x) := \argmin_{y\in\R}\Big\{\frac 12 |x-y|^2 +\eta|y|\Big\},
\ee
which can be used to extend this notion to Hilbert-Schmidt operators.
To that end, we make use of the well-known fact that for any two Hilbert spaces $\cH,\tilde\cH$,
and $v,w\in \cH\otimes \tilde\cH$ with sequences $\sigma^\cH(v),  \sigma^{\cH}(w)$ of singular values
associated with the matricizations $M_\cH(v), M_{\cH}(w)$, one has
\be
\label{singdiff}
\|\sigma^\cH(v) -\sigma^\cH(w)\|_{\ell_2}\le \|v-w\|_{\cH\otimes \tilde\cH} = \|M_\cH(v)-M_\cH(w)\|_{\rm HS}.
\ee
We can then define the {\em nuclear norm} of the matricization $M_\cH(v)$ of $v$ by

\be
\label{nuclear}
\|M_\cH(v)\|_* := \|\sigma^\cH(v)\|_{\ell_1},
\ee
and let 
\be
\label{stH}
S_\eta\big(M_\cH(v)\big) := \argmin_{w\in \cH\otimes\tilde\cH} \Big\{\eta\|M_\cH(w)\|_* + \frac 12\|v-w\|_{\cH\otimes\tilde\cH}^2\Big\}.
\ee

Returning now to the setting \eqref{cH}, with matricizations $M_i(u) = M_{\cH_{[e_i]}}(u)$, $i=1,\ldots,E$,
we denote by $M_i$ and $M_i^{-1}$ the mappings taking a $u\in \cH$ into its matricization $M_i(u)$ 
relative to $[e_i]\in\T_d$ and its inverse, respectively, to set
\be
\label{ith}
\mfS_{i,\eta}(u):= \big(M_i^{-1}\circ S_\eta \circ M_i\big)(u),\quad i=1,\ldots,E,
\ee
and finally
\be
\label{Seta}
\mfS_\eta(u) := \big(\mfS_{E,\eta}\circ\cdots\circ \mfS_{1,\eta}\big)(u).
\ee
It has been shown in \cite{BS:17} that this
inherits the non-expansiveness of the scalar thresholding operator \eqref{nonexp},
\be
\label{Lip1}
\|\mfS_\eta(u)- \mfS_\eta(v)\|_\cH \le \|u-v\|_\cH,
\ee
see \cite[Prop. 3.2]{BS:17}.

\section{Solution Strategies for Operator Equations}\label{sec:solstrat}

In this section, we introduce the precise formulations of our main model problems and their sequence space representations that render them amenable to low-rank tensor approximations. We then describe a common basic construction principle of numerical solvers 
and present the basic principles of deriving   rigorous complexity bounds.

\subsection{Problem Classes and Representative Model Scenarios}\label{sec:probclass}

Let us now consider in more detail the two concrete model problems mentioned in the introduction. \medskip

\noindent\textbf{Diffusion problems (I):} The first concerns diffusion problems in weak formulations on $V:=H^1_0(\Omega)$ where $\Omega = \Omega_1\times \cdots\times \Omega_d$ is a product domain. We wish to find $u\in V$ for given $f \in V'$ such that
\begin{equation}\label{diffusion}
	a(u,v) := \int_\Omega M \nabla u \cdot \nabla v \,dx = f(v), \quad v \in V.
\end{equation}
Here we assume for simplicity that $\Omega_i = [0,1]$ and that $M \in \R^{d\times d}$ is constant and symmetric (with the Poisson problem $M= I$ as a special case). This stationary boundary value problem also constitutes a first step in treating more general evolution or eigenvalue problems on high-dimensional domains.
The bilinear form $a(\cdot,\cdot)$ is ensured to be $V$-elliptic by the assumption that there exist $\gamma, \Gamma>0$ such that $\gamma \leq \langle M \xi, \xi\rangle \leq \Gamma$ for all $\xi \in \R^d$, $\abs{\xi} = 1$. Defining the operator $A$ by $\langle A u, v\rangle = a(u,v)$, $u,v\in V$,
Lax-Milgram's Theorem then says that $A:V\to V'$ is an {\em isomorphism}. 
 
 It will later be important to note that the energy space $V$ is in this case
the {\em intersection} of tensor products of Hilbert spaces
\be
\label{intersect}
V= \bigcap_{j=1}^d L_2(\Omega_1)\otimes\cdots\otimes L_2(\Omega_{j-1}) \otimes H^1_0(\Omega_j)\otimes L_2(\Omega_{j+1})
\otimes\cdots\otimes L_2(\Omega_{d}).
\ee

\noindent\textbf{Parametric Problems (II):} Our second model problem concerns second-order elliptic PDEs with diffusion coefficients depending in an affine manner on scalar parameters $y \in Y := (-1,1)^d$.  Now let $V:= H^1_0(\Omega) \otimes L_2(Y,\mu)$, with $\Omega \subset \R^m$ any domain with, e.g., $m \in \{1,2,3\}$, and let $\mu$ denote  the uniform measure on $Y$. Assuming for simplicity right hand sides $f\in H^{-1}(\Omega)$ independent of $y\in Y$, 
 we consider problems of the form: find $u \in V$
 such that for all $v\in V$
\begin{equation}\label{parametric}
 a(u,v) := \int_{Y}\int_{\Omega}  a(y)\, \nabla u(y) \cdot \nabla v(y)\,dx\,d\mu(y)  = \int_Y \int_\Omega f(v(y)) \,dx\,d\mu(y),
\end{equation}
where $a(y) = \bar a + \sum_{i=1}^d y_i \psi_i$. We assume that $\bar a, \psi_i \in L_\infty(\Omega)$, for $i=1,\ldots,d$, satisfy
$\sum_{i=1}^d \abs{\psi_i} \leq \theta \bar a$ for some $\theta < 1$, which ensures that the bilinear form $a(\cdot,\cdot)$ is  $V$-elliptic and 
the induced operator  $A:V\to V'$ is again an isomorphism. 
In this case the energy space $V = H^1_0(\Omega) \otimes L_2(Y,\mu)$ is a tensor product space endowed with a cross-norm, since
 \[
   L_2(Y,\mu) =   \bigotimes_{i=1}^d L_2\big((-1,1),\textstyle\frac {d y_i}2\big).
 \]

 Here the case $d=\infty$ of countably many parameters, which arises, e.g., in the Karhunen-Lo\`eve expansion of random fields, is 
 explicitly permitted and the fact that target functions may depend on infinitely many variables is a particular challenge in this scenario.

\subsection{Sequence Space Formulation}\label{ssec:seqspace}
The crucial role of rigorous a posteriori error bounds has been emphasized before. Such bounds can be based on the 
{\em error-residual relation} \eqref{res} valid for elliptic problems. Rather than employing duality arguments relying on spatial localization
as for problems in low spatial dimensions, which here are infeasible, we will exploit these relations through first transforming the continuous problem into an equivalent one where domain and range of the transformed operator are the same. More precisely,
as a first step,  to make the considered variational formulations amenable to the concepts discussed in Section \ref{sec:tensors}, we  choose coordinates on the underlying Hilbert spaces via suitable isomorphisms to $\ell_2$ sequence spaces, endowed with the 
norm $\|\bv\|_{\ell_2(\cI)}^2:= \sum_{i\in\cI}|\bv_i|^2$.  Linear mappings of this kind are given by \emph{Riesz bases} of these spaces: a family $\Phi := \{  \varphi_i \}_{i \in \Ical}$ in a Hilbert space $\Hcal$ is Riesz basis if there exist $c_\Phi,C_\Phi >0$ such that 
\be
\label{Riesz}
 c_\Phi \norm{\bv}_{\ell_2(\cI)} \leq \Bignorm{\sum_{i \in \cI} \bv_i \varphi_i }_\Hcal \leq C_\Phi \norm{\bv}_{\ell_2(\cI)},\quad \bv \in \ell_2(\cI), 
\ee
which implies that the corresponding linear mapping $S_\Phi \colon \ell_2(\Ical)\to \Hcal, \bv \mapsto \sum_{i\in\Ical} \bv_i \varphi_i$, is bounded and continuously invertible. Its adjoint is given by $S'_\Phi \colon \Hcal' \to \ell_2(\Ical), \psi \mapsto ( \psi(\varphi_i) )_{i \in \Ical}$. For a detailed description of such Riesz-bases in both scenarios (I) and (II) the reader is referred to \cite{BD2,BD3,BCD15}.

Now suppose one has a Riesz basis $\Phi$ for the energy space $V$ on which the   variational problem:  find $u\in V$, such that
\be
\label{varprob} 
a(u,v)= f(v), \quad v\in V,
\eeqn 
is posed.
Then, for $A:V\to V'$ induced by \eqref{varprob}, defining $\bA = S'_\Phi A S_\Phi$, $\bu = S^{-1}_\Phi u$ and $\bbf = S'_\Phi f$,
\be
\label{l2}
\bA \bu = \bbf,
\ee
is an equivalent reformulation  of \eqref{varprob}  on $\ell_2(\Ical)$. Since $A$ and $S_\Phi$ are isomorphisms,   $\bA: \ell_2(\Ical)\to \ell_2(\Ical)$ is  a boundedly invertible mapping  with the explicit representation $\bA= \big(a(\varphi_\lambda,\varphi_\nu)\big)_{\lambda,\nu\in \cI}$  (see e.g. \cite{actanum}).   
Moreover, it is not hard to verify that 
\be
\label{constants}
 {c_a c^2_\Phi}  \|\bv\|_{\ell_2(\cI)}\le \|\bA\bv\|_{\ell_2(\cI)} \le {C_a C^2_\Phi}  \|\bv\|_{\ell_2(\cI)},\quad \bv \in  \ell_2(\cI).
\ee
In what follows we often abbreviate $\|\cdot\|:=\|\cdot\|_{\ell_2(\cI)}$ if the index domain $\cI$ is clear from the context.

\subsection{Iterative Solvers with Recompression}\label{sec:iter}

Assuming that we have a well-conditioned representation $\bA \bu = \bbf$ on a product sequence space $\ell_2(\Ical)$ with $\Ical = \Ical_1 \times\cdots \times\Ical_d$, we now aim to iteratively construct an approximation of the solution coefficient sequence $\bu$ in hierarchical low-rank format. Due to the high dimensionality of the index set $\Ical$, this is only feasible if all steps in such an iteration are performed entirely on low-rank representations, which requires corresponding low-rank representations (or approximations) of $\bA$ and $\bbf$.

In our present context of elliptic problems, the most straightforward way of obtaining a sequence of approximations converging to $\bu$ is a Richardson iteration: As a consequence of \eqref{constants}, one can find $\omega>0$ such that the iterative scheme
\be
\label{Richardson}
\bu^{n+1}=\bu^n +\omega (\bbf - \bA\bu^n),\quad n=0,1,2,\ldots,
\ee
converges to $\bu$ for any $\bu^0$.
Of course, the iteration \eqref{Richardson} on an infinite-dimensional sequence space cannot be directly realized numerically. 
\smallskip

\noindent\emph{Finite supports:} First, one needs to ensure that all $\bu^n$ have finitely many nonzero entries. Since $\bbf$ and each column of $\bA$ are generally infinitely supported, this amounts to an appropriate truncation of $\bA\bu^n-\bbf$, to arrive at a computable perturbed version
\begin{equation}\label{richardson}
      \bu^{n+1} = \bu^n - \omega_n \br^n, \quad \bu^0 = 0,
\end{equation}
of \eqref{Richardson}.
 The most common strategy is to make an educated guess of a fixed $\Lambda \subset \Ical$ and always use $\br^n = (\bA \bu^n - \bbf)|_\Lambda$, which enforces $\supp \bu^n \subseteq \Lambda$ for all $n$. However, due to the limited accuracy in the residual approximation, the iteration then only converges to the Galerkin approximation $\bu_\Lambda$ given by $(\bA|_{\Lambda\times\Lambda})\bu_\Lambda = \bbf|_{\Lambda}$. Especially in high-dimensional problems, the appropriate choice (or refinement) of such $\Lambda$ to achieve a certain target error is typically not  obvious. Therefore, discretizations will never be fixed beforehand but will be adaptively updated.
\smallskip

\noindent\emph{Tensor ranks:} A second issue is that, with a basic scheme as in \eqref{richardson}, the tensor ranks in the representation of $\bu^n$ may increase rapidly with respect to $n$: in the addition of the low-rank representations of two vectors, their ranks are added, whereas the action of an operator in low-rank form leads to a multiplication by its ranks. Although methods using fixed-rank representations of all iterates can be constructed \cite{Grasedyck:13,KressnerTobler:11}, which essentially attempt to obtain an approximation by optimizing each component in this fixed tensor representation, the high degree of nonlinearity in the resulting problems makes their convergence analysis an extremely delicate problem. In addition, such approaches then still need to be coupled with a procedure for rank adaptation for a given target accuracy.
\smallskip
 
Thus, it is natural to let both the set of activated basis indices and the representation ranks evolve over the course of the iteration and hence to gradually refine both in parallel. 
The methods of this type that have been studied so far share the construction principle of combining mappings $\Fcal_n$, providing an error reduction, combined with mappings $\Rcal_n$ that perform a re-approximation with complexity reduction,
\begin{equation}\label{itertemplate}
   \bu^{n+1} = \mathcal{R}_n(\mathcal{F}_n(\bu^n)).
\end{equation}
Here $\Fcal_n$ can, in principle, be any procedure providing a guaranteed error reduction, such as a fixed-point iteration as in \eqref{Richardson}. Since we want this error reduction to happen with respect to the exact solution $\bu$ in $\ell_2(\Ical)$, in general this necessitates that $\br^n$, and hence $\bu^{n+1}$, have larger support than $\bu^n$, and that hierarchical ranks need to grow during the iteration.
Accordingly, $\Rcal_n$ needs to both eliminate extraneous basis indices and reduce the ranks of iterates.

This entails a compromise between preserving a sufficient error reduction while at the same time preventing too large a growth in the representation costs of the iterates.
Ideally, the reduction operation $\mathcal{R}_n$ should be adjusted to the error reduction so as to ensure convergence of the iteration with (up to a multiplicative constant) the {\em best achievable total computational costs} in terms of the number of operations, that is, to ensure asymptotically optimal complexity.
The basic template \eqref{itertemplate} has been used in the construction of the first adaptive wavelet methods with convergence rates \cite{Cohen:01,Cohen:02}, where optimality was established. 
Procedures of the form \eqref{itertemplate} are also a core ingredient in many iterative methods operating on low-rank representations, for instance those proposed in \cite{Beylkin:02,Hackbusch:08,Khoromskij:11-1,KressnerTobler:11,Ballani:13,BNZ:14}.   
The choices of $\Rcal_n$ employed in these contributions, corresponding to truncation of to fixed ranks or to variable ranks with ad-hoc tolerances, however, do not ensure a suitable compromise between convergence and complexity. 

In Section \ref{sec:complexity}, we consider in detail low-rank solvers with complexity bounds which ensure that this compromise is met
for both scenarios (I) and (II). A crucial role is played by a suitably abstracted version of a ``Coarsening Lemma''  that appeared first
in \cite{Cohen:01} in the context of adaptive wavelet methods.
At this point, we next discuss briefly some common aspects of choosing $\Rcal_n$ and $\Fcal_n$ that can be used to ensure convergence of $\bu^n$ to $\bu$. By the above choice of reference basis functions, this  is equivalent to convergence of the method to the exact solution $u$ in the energy space $V$. This requires on the one hand the identification of a sequence of finite subsets $\Lambda^n = \Lambda_1^n \times \cdots \times \Lambda_d^n$ such that $\supp (\bu^n) \subseteq\Lambda^n$; and on the other hand, for a dimension tree assumed to be given, finding hierarchical ranks and constructing representation coefficients of a hierarchical tensor representation of $\bu^n$.

Two basic ways of constructing $\mathcal{F}_n$ in \eqref{itertemplate} have been considered in the literature: the first is a perturbed iteration \eqref{richardson}, where for any given index set $\Lambda^n$ and ranks $\rr{r}^n$ of $\bu^n$, one needs to provide a routine that can produce finitely supported $\br^n$ in hierarchical format such that $\norm{\br^n - (\bA \bu^n-\bbf)}\leq \eta$ for any $\eta >0$. For appropriate choices of $\eta$ depending on $n$, this generally requires an enlarged product set $\Lambda^{n+1} = \Lambda_1^{n+1} \times\cdots\times \Lambda_d^{n+1}\supset\Lambda^n$ to satisfy $\supp(\br^n),\supp(\bu^{n+1})\subseteq \Lambda^{n+1}$.
Moreover, the tensor representation of $\bu^{n+1}$ resulting from performing \eqref{richardson} in low-rank format generally has larger ranks $\rr{r}^{n+1}$.
This strategy is considered in \cite{BD,BD2,BD3,BCD15}.

The second basic construction of $\mathcal{F}_n$, which is conceptually closer to adaptive finite element methods, uses sequential Galerkin solves: 
for given $\Lambda^n$, define $\bu^n$ as the corresponding Galerkin solution, $\bu^n := \bu_{\Lambda^n}$. Then, find a product set $\Lambda^{n+1} \supseteq \Lambda^n$ such that $\norm{ (\bA\bu^n - \bbf)|_{\Lambda^{n+1}}} \geq \tau \norm{ \bA\bu^n - \bbf }$ for a fixed $\tau \in (0,1)$. This can be achieved by a sufficiently accurate finitely supported approximation $\br^n$ as in the case of \eqref{richardson}. The process of approximating the Galerkin solution $\bu_{\Lambda^{n+1}}$, e.g., by a Krylov space method, then leads to new hierarchical representation ranks $\rr{r}^{n+1}$.
This strategy has been analyzed in \cite{AU:18}.

Accordingly, $\Rcal_n$ is chosen as a composition of a rank reduction as discussed in Section \ref{ssec:properties}, for instance by hard thresholding or soft thresholding of hierarchical singular values, and a subsequent reduction of active basis indices. 
 The concrete realizations of $\Fcal_n$ and $\Rcal_n$ depend on the type of the problem \eqref{varprob}. They rely on two
 essential ingredients discussed next.

 \subsection{Two Core Ingredients}\label{ssec:core}

The main distinction between the two problem scenarios (I) and (II) lies in the structure of the respective energy space $V$.
This has an essential effect on the approximate evaluation of $\bA\bu^n$.
In the case (II) where the space on which the variational problem is posed is a tensor product space endowed with a cross norm, such as $V = H^1_0(\Omega)\otimes\bigl( \bigotimes_{i=1}^d L_2((-1,1),\frac12 dy_i) \bigr)$ in \eqref{parametric}, a choice of Riesz bases $\Phi_0, \Phi_1,\ldots,\Phi_d$ in each of the $(d+1)$ factors leads to a {\em tensor product Riesz basis} $\Phi$ of $V$ (indexed, say, by $\N^{d+1}$), and hence $S_\Phi = S_{\Phi_0} \otimes \cdots S_{\Phi_d}$. As a consequence, low-rank structures in $A$ and $u$ are preserved in their representations $\bA\colon \ell_2(\N^{d+1})\to \ell_2(\N^{d+1})$ and $\bu \in \ell_2(\N^{d+1})$; for instance, both $A$ in \eqref{parametric} and the corresponding $\bA$ can be written as sums of $d+1$ Kronecker products. We then aim to find, for instance, a low-rank representation of the coefficient sequence $\bu$ as a tensor of order $d+1$ in hierarchical low-rank format.

\subsubsection{Finite rank adaptive scaling}
In the case (I) of \eqref{diffusion},  the mapping $S_\Phi$ can, in general, {\em not} be chosen to be of Kronecker rank one. To see this, it is instructive to consider the case $A= -\Delta$ on $\Omega = (0,1)^2$ in \eqref{diffusion}. A simple choice of basis can be derived from the $L_2$-normalized eigenfunctions $\tilde\varphi_{k_1,k_2}$ of $A$ given by $\tilde\varphi_{k_1,k_2}(x_1,x_2) = 2 \sin(\pi k_1 x_1) \sin(\pi k_2 x_2)$ for $k_1, k_2 \in \N$, with corresponding eigenvalues $\pi^2(k_1^2 + k_2^2)$. The functions $\varphi_{k_1,k_2}(x_1,x_2) = 2 \pi^{-2} (k_1^2 + k_2^2)^{-1/2}  \sin(\pi k_1 x_1) \sin(\pi k_2 x_2)$ are thus an orthonormal basis (that is, a Riesz basis with $c=C = 1$) of $V = H^1_0(\Omega)$. The resulting mapping $S_\Phi$ for $\Phi = \{ \varphi_{k_1,k_2} \}_{k\in\N^2}$ then cannot be written as a finite sum of Kronecker products due to the presence of the factor $(k_1^2 + k_2^2)^{-1/2}$. This reflects the fact that $H^1_0(\Omega)$ is not endowed with a cross norm, but can rather be characterized as the \emph{intersection} of product spaces $H^1_0(0,1)\otimes L_2(0,1)\cap L_2(0,1)\otimes H^1_0(0,1)$ with its natural norm.

As shown by \eqref{intersect}, the situation for larger $d$ and more general Riesz bases is analogous, i.e., Riesz bases for
the energy space $V$   do no longer consist of separable functions.  As a result the operator $\bA$ has in this case {\em infinite rank}.
In order to still facilitate in the end
an efficient error-controlled approximate application of the operator representation $\bA$ requires a judicious {\em preconditioning strategy} 
that realizes a proper balance between rank growth and stability needed to maintain convergence in \eqref{itertemplate}.

To that end,  a particularly well suited building block are {\em tensor product wavelet bases}, constructed from an $L_2$-orthonormal wavelet basis $\{ \psi_\lambda \}_{\lambda \in \vee}$ on a one-dimen\-sional domain $I_1 \subseteq \R$, assuming $\Omega := I_1\times\cdots\times I_1$. Here $\vee$ is an appropriate set of scale-space indices, where the level of $\lambda \in \vee$ is denoted by $\abs{\lambda}$. Assuming sufficient regularity of the $\psi_\lambda$, with $\Ical := \bigtimes_{i=1}^d \vee$ and $\Psi_{\lambda} := \psi_{\lambda_1}\otimes\cdots\otimes \psi_{\lambda_d}$, one has that $\Phi := \{ \norm{\Psi_{\lambda}}_{H^1_0}^{-1}\Psi_{\lambda} \}_{\lambda=(\lambda_1,\ldots,\lambda_d) \in \Ical}$ is a Riesz basis of $H^1_0(\Omega)$. Here one has the proportionality $\norm{\Psi_{\lambda}}_{H^1_0}^{-1} \sim \omega_\lambda := (2^{2\abs{\lambda_1}} + \ldots + 2^{2\abs{\lambda_d}})^{-1/2}$ uniformly in $\lambda$, which leads to a similar lack of separability of $S_\Phi$,
see e.g. \cite{Cohen:01,SSprod:08}.

Instead, the representation $\bT := S_\Psi' A S_\Psi$ of $A$ with respect to the $L_2$ orthonormal basis $\Psi$ does have 
  a simple low-rank structure.  For instance, in the case of the Laplacian, where $\Delta = \partial_{1}^2\otimes I + I\otimes \partial_2^2$ in the above example with $d=2$.  However, $\bT$ is not bounded on $\ell_2(\cI)$ and hence cannot be used 
  in \eqref{Richardson} or its perturbed version \eqref{richardson}. The low-rank structure 
  is unfortunately lost   in the boundedly invertible representation 
 \be
 \label{ST}
 \bA = \bS^{-1}\bT\bS^{-1},\quad (\bS)_{\lambda,\lambda'}= \delta_{\lambda,\lambda'} (2^{2\abs{\lambda_1}} + \ldots + 2^{2\abs{\lambda_d}})^{-1/2},\,\,\lambda \in \cI,
 \ee
 after an ideal scaling that
 takes the topology of $V$ into accout and ensures the necessary mapping property of $\bA$ relating errors to residuals. 
 
 A crucial point for what follows, however, is that in the case $V= H^1_0(\Omega)$, the departure from the original low-rank structure is still controllable in the sense that $S_\Phi$ has efficient and explicitly computable low-rank approximations.
One class of such approximations is provided by {\em exponential sum approximations} of the ideal non-separable scaling factors $\omega_\lambda$.
This is exemplified by the following result, paraphrased from \cite[Thm.\ 9]{BD2}.

\begin{theorem}
\label{thm:expsumrel}
There exist positive sequences $(w_j)_{j\in\N}$, $(t_j)_{j\in\N}$ such that for all $\lambda \in \Ical =\bigotimes_{i=1}^d \vee$,
\begin{equation}\label{expsumbasic}
   \tilde\omega_\lambda^{(\infty)} := \sum_{j=1}^\infty w_j \prod_{i=1}^d e^{-t_j 2^{2 \abs{\lambda_j}}} < \infty
\qquad\text{satisfies}\qquad
  \bigabs{ \omega_\lambda - \tilde\omega_\lambda^{(\infty)} }  \leq \frac{\omega_\lambda}{2} ,
\end{equation}
and in addition, there exists $C>0$ such that with
\be
\label{tildeomega}
\tilde\omega_\lambda^{(m)} := \sum_{j=1}^m w_j \prod_{i=1}^d e^{-t_j 2^{2 \abs{\lambda_i}}},
\ee
 one has
\begin{equation}\label{expsumlrapprox}
  \bigabs{ \omega_\lambda - \tilde\omega_\lambda^{(m)}  } \leq \frac{\omega_\lambda}{2} , 
  \quad \bigabs{ \tilde\omega_\lambda^{(\infty)} - \tilde\omega_\lambda^{(m)}  } \leq \eta \omega_\lambda, \qquad \lambda \in \Lambda,
\end{equation}
for any finite $\Lambda \subset \Ical$, provided that 
\be
\label{m}
m \geq  C \bigl(1 + \abs{\log \eta} + \max_{\lambda \in \Lambda} \abs{\log \omega_\lambda} \bigr).
\ee
\end{theorem}

By \eqref{expsumbasic}, $\tilde\Phi := \{ \tilde\omega^{(\infty)}_\lambda \Psi_\lambda \}_{\lambda \in \Ical}$ is then also a Riesz basis of $H^1_0(\Omega)$, with constants $c_{\tilde\Phi},C_{\tilde\Phi}$ independent of $d$. For finite subsets of $\Ical$, the corresponding mapping $S_{\tilde\Phi}$ has efficient low-rank approximations whose Kronecker rank increases logarithmically in the desired error and linearly in 
\[
  \max_{\lambda \in \Lambda} \abs{\log \omega_\lambda} \sim \max_{\substack{j=1,\ldots,d\\ \lambda \in \Lambda}} \abs{\lambda_j}\,.
  \]
This quantity corresponds roughly to the largest frequency of an activated wavelet basis element. For problems on $H^1_0(\Omega)$, we can thus use the sequence space formulation $\bA \bu = \bbf$ with
\be
\label{A1}
  \bA = S_{\tilde\Phi}' A S_{\tilde\Phi} =: \tilde \bS^{-1}_\infty \bT \tilde \bS^{-1}_\infty, \quad \bu = S^{-1}_{\tilde\Phi} u, \quad \bbf = S'_{\tilde \Phi} f,
\ee
where $\tilde \bS^{-1}_\infty= {\rm diag} (\tilde\omega_\lambda^{(\infty)})_{\lambda\in\cI}$.
Then, with $\tilde\Phi_m := \{ \tilde\omega^{(m)}_\lambda \Psi_\lambda \}_{\lambda \in \Ical}$, the infinite matrix $\bA^{(m_1,m_2)} := S_{\tilde\Phi_{m_1}}' A S_{\tilde\Phi_{m_2}}$ provides an explicit low-rank approximation of $\bA$ when $m_1$, $m_2$ are chosen appropriately according to \eqref{expsumlrapprox}. For instance, in \eqref{diffusion} with $M = I$, using $L_2$-orthonormality of the $\psi_{\lambda_i}$, we have
\begin{align}
&  (\bA^{(m_1,m_2)})_{\lambda,\mu} =\!\! \sum_{j_1=1}^{m_1}\! \sum_{j_2=1}^{m_2} \!w_{j_1} w_{j_2} \!\sum_{k=1}^d e^{-t_{j_1} 2^{2 \abs{\lambda_{k}}} - t_{j_2} 2^{2 \abs{\mu_{k}}}} \langle \psi'_{\lambda_i}, \psi'_{\mu_i} \rangle\nonumber\\
& \qquad\qquad\qquad\qquad \prod_{i \neq k} e^{-t_{j_1} 2^{2 \abs{\lambda_{i}}} - t_{j_2} 2^{2 \abs{\mu_{i}}}} \delta_{\lambda_i,\mu_i}.
\end{align}
Note that although this gives $d m_1 m_2$ terms of Kronecker rank one, $\bA^{(m_1,m_2)}$ can in fact be represented as an operator in hierarchical tensor format with hierarchical ranks bounded by $2 m_1 m_2$.
One can proceed similarly for $\bbf$, and the approximate coefficients of $u$ with respect to the $L_2$-basis\footnote{Schemes for performing SVD-type approximations in $H^1$ directly using $L_2$-basis coefficients have recently been proposed in \cite{AN:18}; however, this leads to weakened quasi-optimality properties.} $\{ \Psi_\lambda\}_{\lambda\in \Ical}$ can be recovered by applying $S_{\tilde\Phi_{m}}$ with sufficiently large $m$ to any finitely supported approximation of $\bu$.

In summary, in problems such as \eqref{diffusion}, one arrives at a coupling between the costs of approximate low-rank representations of $\bA$ and the activated finite subsets of basis indices. In terms of the discretizations defined by these subsets, this means that preconditioning the problem in low-rank form to arrive at a well-conditioned matrix equation becomes more costly in terms of the arising ranks as the discretization is refined. How to dynamically adjust the choice of $m_1,m_2$ for given target tolerances $\eta$ will be taken up again in Section \ref{sec:complexity}.
This issue does not arise in scenario (II), since there $V$ is endowed with a cross-norm.

\subsubsection{Contractions and tensor coarsening}\label{ssec:reduction}
Thus, the reduction operator $\Rcal_n$ in \eqref{itertemplate} must involve, aside from rank-reduction, a mechanism to
control represenation complexity in terms of activated basis functions, which is the second core  component in all varients 
of \eqref{itertemplate}. Specifically,
this requires   the identification of a sequence of finite subsets $\Lambda^n = \Lambda_1^n \times \cdots \times \Lambda_d^n$ such that $\supp (\bu^n) \subseteq\Lambda^n$ that warrants sufficient accuracy.
 In purely sparsity-based adaptive methods, as introduced in \cite{Cohen:01,Cohen:02}, such a reduction of indices can be done by keeping the entries of largest absolute value. In our context, this is not an option: on one hand, directly sorting the high-dimensional coefficient sequence is not possible for complexity reasons; on the other hand, the tensor decomposition requires the product structure of index sets to be preserved. A possible way around these two restrictions, proposed in \cite{B} and analysed further in \cite{BD,BD2}, is to select the relevant basis indices based on the {\em one-dimensional sequences}
\[
 \pi^{(i)}(\mathbf{v}) = \bigl( \pi^{(i)}_{\lambda_i} (\mathbf{v})
 \bigr)_{\lambda_i\in\Ical_i}  
  \in \ell_2(\Ical_i),
\]
where
\[
  \pi^{(i)}_{\lambda_i} (\mathbf{v}) := \Bigl(\sum_{%
 \lambda_1,\ldots,\lambda_{i-1},\lambda_{i+1},\ldots,\lambda_d} \abs{v_{ \lambda_1,\ldots,\lambda_{i-1},\lambda_i, \lambda_{i+1},\ldots,\nu_d}}^2 \Bigr)^{\frac{1}{2}} ,
\]
which were termed \emph{contractions} in \cite{BD} and can also be regarded as one-dimensional densities of the tensor.
These can be efficiently evaluated using the identity\footnote{A technique of using \eqref{eq:contreval} for steering adaptivity in high dimensions was also subsequently developed independently in \cite{EPS17}.}
\begin{equation}\label{eq:contreval}
 \pi^{(i)}_\lambda (\mathbf{v}) = \Bigl( \sum_{k}
 \bigabs{\mathbf{U}^{(i)}_{\lambda, k}}^2 \bigabs{\sigma^{(i)}_{k}}^2
 \Bigr)^{\frac{1}{2}}  ,\quad \lambda \in \Ical_i,
\end{equation}
in terms of the ${H}$SVD mode frames $\mathbf{U}^{(i)}$ and the corresponding sequences of singular values $\sigma^{(i)}$.
One easily obtains the following {\em quasi-optimality} result for coarsening based on these sequences.

\begin{proposition}
\label{prop:coarsen}
Let $\Lambda(\bv; N) = \Lambda^{(1)}(\bv;N)\times \cdots\times  \Lambda^{(d)}(\bv;N) \subset \Ical$ be the product index set corresponding to the $N$ largest elements of $\{ \pi^{(i)}_\lambda(\bv) \colon i=1,\ldots,d, \; \lambda \in \Ical_i\}$. Then for any $\hat\Lambda = \hat\Lambda_1\times\cdots\times\hat\Lambda_d$ with $\sum_{i=1}^d \#\hat\Lambda_i \leq N$, with
\begin{equation}\label{tensorcoarsen_remainder}
	s_{N}(\bv) := \biggl( \sum_{i=1}^d \sum_{\lambda \in \Ical_i \setminus \Lambda^{(i)}(\bv;N)} \abs{\pi^{(i)}_\lambda (\bv) }^2 \biggr)^{\frac12}
\end{equation}
one has
\[
  \norm{ \bv - ( \bv|_{\Lambda(\bv;N)}) } \leq s_{N}(\bv) \leq \sqrt{d} \norm{ \bv - (\bv|_{\hat\Lambda} ) } . 
\]
\end{proposition}

Note that this result is analogous to Theorem \ref{thm:hier_tensor_recompression_est}, and as discussed in further detail in Section \ref{sec:complexity}, the effect of the two types of reduction operations combined in $\mathcal{R}_n$ can be estimated using essentially the same techniques.
For finitely supported input sequences $\bw$ on $\Ical$, in both cases one can implement computational routines with analogous properties:
\begin{itemize}
	\item $\recompress(\bw;\eta)$, which performs ${H}$SVD hard thresholding, such that $\bw_\eta:= \recompress(\bw;\eta)$ has hierarchical ranks $\rr{r}_\eta:=\rr{r}_\mathbb{E}(\bw_\eta)$ with $\max_i r_{\eta,i}(\bw_\eta)$ \emph{minimal} such that $t_{\rr{r}_\eta}(\bw) \leq \eta$, and hence by Theorem \ref{thm:hier_tensor_recompression_est}, $\|\bw - \bw_\eta \|_{\ell_2(\cI)}\le \eta$. This routine has the cost of performing the ${H}$SVD as discussed in Remark \ref{rem:compl}. One has the quasi-optimality property
\[
   \|\bw- \coarsen(\bw;\eta)\| \leq \sqrt{2d-3} \min_{\rr{r}_\mathbb{E}(\hat\bw) \leq {\rr{r}_\eta}} \norm{ \bw - \hat\bw } \,.
\]
	
	\item $\coarsen(\bw;\eta)$, which determines the product index set $\Lambda(\bw;N_\eta)$ as in Proposition \ref{prop:coarsen} with $N_\eta$ minimal such that $s_{N_\eta}(\bw) \leq \eta$, which ensures $\|\bw- \coarsen(\bw;\eta)\| \le \eta$. This requires computing $\pi^{(i)}_\lambda(\bw)$ from the ${H}$SVD for all $i=1,\ldots,d$, $\lambda \in \Ical_i$, and (quasi-)sorting these values. From Proposition \ref{prop:coarsen}, one has the quasi-optimality property
\[
  \|\bw- \coarsen(\bw;\eta)\| \leq \sqrt{d} \min_{\hat\Lambda} \norm{ \bw - (\bw|_{\hat\Lambda} ) } ,
\] 
where the minimum is over all $\hat\Lambda = \hat\Lambda^1\times\cdots\times\hat\Lambda^d$ such that $\sum_i\#\hat\Lambda_i\leq N_\eta$.
	
\end{itemize}

An important feature of the resulting schemes of the form \eqref{itertemplate}, with appropriately adjusted $\Fcal_n$ and $\Rcal_n$, is their \emph{universality:} they require no information on the approximability of the sought solution $\bu$ to produce quasi-optimal solutions.
In order to assess the performance of such methods, however, we hypothetically assume $\bu$ to have certain approximability properties and then show that the method under consideration will have the expected complexity in such a case without using, however,
any a priori knowledge of such approximability properties.

\subsection{Approximation Classes}\label{ssec:apprcl}

The natural approximability properties to consider are strongly tied to the choice of $\Rcal_n$, that is, to the routines
$\recompress$ and $\coarsen$; in our case, these properties are the following intrinsic features of $\bu$:
\begin{enumerate}[(i)]
\item The convergence rate of the best approximation in $\ell_2(\cI)$ of hierarchical rank $\rr{r}$ to $\bu$ as $\rr{r}$ increases, for instance, as $\abs{\rr{r}}_\infty\to \infty$.
\item The asymptotic decay of the decreasing rearrangement of all values $\pi^{(i)}_\lambda(\bu)$ for $i=1,\ldots,d$ and $\lambda \in \Ical_i$.
\end{enumerate}

The notion of approximation classes  is a classical means to quantify approximation properties. 
Regarding (i), we follow \cite{BD} and
consider positive and strictly
increasing sequences  $\gamma = \bigl(\gamma(n)\bigr)_{n\in \N_0}$ with $\gamma(0)=1$ and $\gamma(n)\to\infty$ as $n\to\infty$.
For $\bv \in \spl{2}(\Ical)$ let
\[
  \norm{\bv}_{\AH{\gamma}} \coloneqq \sup_{r\in\N_0} \gamma({r})\,\inf_{\max{\rr{r}_\mathbb{E}(\bw)}\leq r} \norm{\bv - \bw} ,
\]
where $\rr{r}_\mathbb{E}(\bw)$ reduces to $\rank(\bw)$ when $d=2$ and $\Ical = \Ical_1\times \Ical_2$, as in the case of separating spatial and parametric variables in the problem class (II). It will be seen in the next section that relevant sequences $\gamma$
cover a wide range from algebraic to (sub-)exponential rates.
This gives rise to the approximation classes
\be
\label{rclass}
\AH{\gamma} \coloneqq  \bigl\{\bv\in \spl{2}(\sidx\times\pidx) : 
 \norm{\bv}_{\AH{\gamma}} <\infty \bigr\} .
\ee
For $\bu \in \AH{\gamma}$, a hierarchical tensor best approximation is guaranteed to achieve accuracy $\varepsilon$ with hierarchical ranks $\rr{r}_\varepsilon$ whenever
\be
\label{maxr}
   \max \rr{r}_\varepsilon \geq \gamma^{-1}\big( \norm{\bv}_{\AH{\gamma}}\,{\e}^{-1}\big).
\ee
We will be especially interested in exponential-type decay of low-rank approximation errors, where with $\gamma(r) = e^{c r^\beta}$ for some $c,\beta$, for $\bv \in \AH{\gamma}$ one has
\[
  \inf_{\max{\rr{r}_\mathbb{E}(\bw)}\leq r} \norm{\bv - \bw}  \leq \norm{\bv}_{\AH{\gamma}} e^{-c r^\beta}
\] 
and accordingly
\[
  \max \rr{r}_\varepsilon \geq \bigl(c^{-1} \ln ( \norm{\bv}_{\AH{\gamma}}\, \varepsilon^{-1})\bigr)^{1/\beta}.
\]

Regarding (ii), for an approximation $\bv$ of bounded support to $\bu$, the number of nonzero coefficients $\#\supp_i \bv$ required in each tensor mode to achieve a certain accuracy depends on the best $n$-term approximability of the sequences $\pi^{(i)}(\bu)$.
 This approximability by sparse sequences is quantified by the classical {\em approximation classes} $\As = \As(\mathcal{J})$, where $s>0$ and $\mathcal{J}$ is a countable index set.  $\As$ is comprised of all $\bw\in \ell_2(\mathcal{J})$
for which the quasi-norm
\begin{equation}
\label{Asdef}
  \norm{\mathbf{\bw}}_{\As(\mathcal{J})} := \sup_{N\in\N_0} (N+1)^s
  \inf_{\substack{\Lambda\subset \mathcal{J}\\
  \#\Lambda\leq N}} \norm{\bw - \Restr{\Lambda} \bw}
\end{equation}
is finite. The approximation of a sequence $\bu \in \As$ up to an error $\varepsilon$ by a sequence with $N_\varepsilon$ nonzero entries is ensured with $N_\varepsilon \geq  \norm{\mathbf{\bw}}_{\As }^{\frac1s} \e^{-\frac1s}$. The separate sparsity of tensor mode frames for $\bu \in \ell_2(\Ical)$ can be quantified by assumptions of the form $\pi^{(i)}(\bu) \in \As(\Ical_i)$ for $i=1,\ldots,d$.

Before turning to the complexity analysis, we next give an overview of available approximability results that indicate
which approximation classes solutions to both problem types (I) and (II) can be expected 
to belong to. This will then serve  formulating proper benchmarks for the performance of  specifications of algorithms.

\section{Approximability}
\label{sec:approximability}

  In this section we discuss approximability properties of solutions to problems  
of both types (I) and (II) with repspect to  low-rank or hierarchical tensor methods partly in comparison with
best $n$-term approximation. These theoretical questions concerning the ``expressive power'' of certain approximation formats serve two purposes.
On one hand, they help judging which approximation type is best suited for which problem class. 
It is clear from the preceding discussions that low-rank or hierarchical tensor formats, due to their ``higher level of nonlinearity'' may increase the computational burden and also raise coding challenges. This should be justified
by a superior performance in terms of the resulting {\em work-accuracy balance}. 

On the other hand,
describing best possible approximability, these results provide {\em benchmarks} for assessing the performance of
versions of Algorithm \ref{alg:tensor_opeq_solve}.
 This  asks for the computational
cost required to meet   the target accuracy in a given format. It includes two basic questions, namely
\begin{itemize}
	\item 
What  is the {\em representation complexity} of a given approximation format for a given problem class, i.e., the number of degrees
of freedom required by a given approximation format to achieve the target accuracy?
\item What is the corresponding {\em computational complexity}, that is, can one devise a numerical method that certifiably achieves a given target accuracy at 
at the expense of a computational work that stays as close as possible   to the representation complexity? 
\end{itemize}
In this section we address the first question mainly for low-rank and tensor methods. 
This asks for a {\em regularity theory} in a broad sense where classical smoothness measures are
replaced by quantified approximability. Corresponding notions of {\em low-rank approximability}
would delineate the extent to which tensor methods could ideally mitigate the curse of dimensionality.
Such methods  exhibit a significantly higher level of nonlinearity than
classical best $N$-term approximation, but are more restricted than DNNs for which comparable theoretical
guarantees are not yet available.

\subsection{High-Dimensional Diffusion}
The representation complexity  for problems of the type \eqref{diffusion} is perhaps best understood 
 for the special case that the diffusion matrix $M$ is diagonal \cite{DDGS}. In fact, what matters most is the following structure
of the operator $A$ induced by the $V$-elliptic  bilinear form in this case, namely  that
\be
\label{Delta}
A = \sum_{j=1}^d I\otimes \cdots I \otimes A_j\otimes I \otimes \cdots  \otimes I,
\ee
has rank $d$ and
  the ``low-dimensional'' operators 
\be
\label{Ai}
A_i : \HH_i \to \HH_i',\quad \langle A_i u,v\rangle = a_i(u,v): \HH_i\times \HH_i \to \R, %
\ee
are $\HH_i$-elliptic. The simplest example is $A_j= - \partial_{x_j}(M_j\partial_{x_j})$ and $\HH_j= H^1_0(\Omega_j)$, $\Omega_j\subset \R^m$,
$j=1,\ldots,d$. As seen below, it is not important that the $A_j$ are second order operators. Again the energy space $V$ has then the form
\eqref{intersect} with $H^1_0(\Omega_i)$ replaced by $\HH_i$.

While the structure of $A$ may nourish hope that the solution to 
\be
\label{opA}
Au = f \quad \text{on }\, \Omega_1\times \cdots \times \Omega_d
\ee
can be accurately approximated by low-rank functions, $A^{-1}$ as a mapping from $V'$ to $V$ has infinite rank which
renders the question of quantifiable low-rank approximability less clear.
In view of \eqref{res}, it is this mapping property that determines the representation complexity.
We review in this section some results from \cite{DDGS} on quantified {\em low-rank approximability} of solutions
to high-dimensional diffusion equations of the form \eqref{opA}, \eqref{Delta}  that do establish rigorous relations between target accuracies 
and rank growth in a ``tamed canonical tensor format'', as explained below. While the canonical format is most convenient 
for the present purpose, it can be embedded in the hierarchical tensor format as well, see \cite{BSU:18}. 
The term ``tamed'' accounts for the fact that computing in this format is, in general, prone to instability.

The relevant structural properties of $A$ are conveniently reflected by its {\em eigensystem}.
Denoting by $\{e_{j,k}\}_{k\in\N}$ the set of eigenfunctions of the low-dimensional component operators $A_j$
with eigenvalues $\lambda_{j,k}$, one easily verifies that 
\be
\label{Aeigen}
A e_\nu = \lambda_\nu e_\nu,\quad  e_\nu = e_{1,\nu_1} \otimes \cdots \otimes e_{d,\nu_d},\,\,\lambda_\nu = \lambda_{1,\nu_1}+\cdots +\lambda_{d,\nu_d}.
\ee
This allows one to define the smoothness norm
\be
\label{ns}
\|v\|_s^2:= \sum_{\nu\in \N^d} \lambda_\nu^s|\langle v,e_\nu\rangle|^2.
\ee
Denoting by $\HH^t$ the space of all $v$ for which $\|v\|_t$ is finite,
the operator $A$ then acts as isometries between these scales, $\|A\|_{\mathcal{L}(\HH^{t+2},\HH^t)} = 1$, $t\in \R$.
While these norms exist for all $t\in \R$, they are equivalent with classical Sobolev norms only for a limited range of smoothness
around $t=1$, depending on the specific problem data. In particular, one always has $V= \HH^1$.

Since obviously $\lambda_\nu^{-1}$ is not separable $A^{-1}: \HH^t\to\HH^{t+2}$ has infinite rank. 
Hence, a bounded rank of data $f$ does not a priori
indicate which rank suffices for an approximate solution to warrant a given accuracy.

We will state next in more precise terms what we mean by {\em low-rank approximability}
of a function. Since we will deal with tensors in
the CP format, continuing to denote rank-one functions as $(g_1\otimes \cdots \otimes g_d)(x):= g_1(x_1)\cdots g_d(x_d)$,  we have
to incorporate means to ensure stability. To that end, we  consider
for $r\in\N$ {\em tamed} rank-$r$ tensor classes of the form
$$
\Tcal_r(\HH^t) := \Big\{ g\in \HH^t : g= \sum_{k=1}^r \bigotimes_{j=1}^d g_j^{(k)}\Big\}.
$$
endowed with the ``norm-like'' quantities
\be
\label{normt}
\tripnorm{g}_{r,t}:= \inf_{ \sum_{k=1}^r \bigotimes_{j=1}^d g_j^{(k)}=g}\Big\{ \max_{k=1,\ldots,r} \{\|g\|_t,\|g^{(k)}\|_t\}\Big\}.
\ee
Controlling  $\tripnorm{g}_{r,t}$ is to account for the fact that rank-$r$ representations are in general not unique and to
avoid cancellation effects between terms with large norms.

Despite the otherwise inherent instability of the CP format this allows us to introduce {\em approximation classes} collecting those elements in 
$\HH^t$ that are close to low-rank tensors
in a quantifiable way. To that end, consider for any $\zeta >0$  the {\em K-functional}
\be
\label{K}
K_r(v,\delta) := K_r(v,\delta,\HH^t, \HH^{t+\zeta}) := \inf_{g\in \Tcal_r(\HH^{t+\delta})}\Big\{\|v- g\|_t +\delta\tripnorm{g}_{r,t+\zeta}\Big\}.
\ee
Then $\delta^{-1}K_r(v,\delta)\le   C$ means that there exists a  rank-$r$ tensor $g$ of smoothness $t+\zeta$ for which $\|v-g\|_t\le \delta C$ and $\tripnorm{g}_{r,t+\zeta} \le C$. To relate the achievable accuracy $\delta$ to the rank $r$, consider as before a strictly increasing  sequence of positive numbers
 $\bgamma = (\gamma(n))_{n\in \N}$,  called a {\em growth sequence}.  Monotonicity of $\gamma$ ensures that the inverse $\gamma^{-1}$ exists.
We now collect in the {\em approximation class}  $\Acal^{\gamma} = \Acal^{\gamma} (\HH^t,\HH^{t+\zeta})$ all those $v\in \HH^t$
for which
the quasi-norm
\be
\label{Anorm}
\|v\|_{\cA^{\gamma}(\HH^t,\HH^{t+\zeta})} := \sup_{r\in \N_0} \gamma(r) K_r(f,\gamma(r)^{-1};\HH^t,\HH^{t+\zeta}),
\ee
is finite, i.e., those elements in $\HH^t$ which can be stably approximated by rank-$r$ tensors in $\HH^{t+\zeta}$ with accuracy
$O(\gamma(r)^{-1})$. In other words,  membership of a $v$ to $\Acal^\gamma$ means that
approximating a given $v\in \Acal^\gamma$ by a tensor in $\HH^{t+\zeta}$ within accuracy $\e >0$, is guaranteed, in analogy to
\eqref{maxr}, by the  rank 
\be
\label{reps}
r=r(\e) = \lceil \gamma^{-1}(\|v\|_{\Acal^\gamma}/\e)\rceil.
\ee

Since $A^{-1}$ has infinite rank one cannot expect any low-rank approximability of the data $f$ to be exactly inherited by the solution
 $u=A^{-1}f$. The following result says that a low-rank approximation order $O(1/\gamma(r))$ for the data $f$ implies a slightly weaker order 
 $O(1/\hat\gamma(r))$ for the solution $u$. The attenuated growth sequence $\hat\gamma(r)$ is given as follows. Let 
 $G(x)= x(\ln \gamma(x))^2$, $x>0$, and define
 \beqn
\label{hatg}
\hat\gamma := \gamma(G^{-1}(m)).
\eeqn
Since $G$ has super-linear growth, $G^{-1}$ has only sublinear growth causing $\hat\gamma(n)$ to have a somewhat weaker growth
than $\gamma$.

\begin{theorem}
\label{thm:ts1}
Assume that 
the data $f$ belong to $\Acal^\gamma(\HH^t,\HH^{t+\zeta})\subset \HH^t$
for some $0< \zeta \le 2$. Then, for $\hat\gamma$ defined by \eqref{hatg}, %
the solution $u$ of \eqref{opA} belongs to $\Acal^{\hat\gamma} (\HH^{t+2},\HH^{t+2+\zeta})\subset \HH^{t+2}$ and
for each $r\in \N$ there exists a   $u_r\in \Tcal_r(\HH^{t+2+\zeta})$ such that
\be
\label{uappr}
\|u - u_r\|_{t+2} \le C\hat\gamma(r)^{-1}\|f\|_{\Acal^\gamma(\HH^t,\HH^{t+\zeta})}, %
\ee
holds for some fixed constant $C$ independent of $r$.
\end{theorem}
The relation between $\gamma$ and  $\hat\gamma$ is illustrated
by taking
$\gamma(r)= r^\alpha$, for some positive $\alpha$. Then $\hat\gamma(r)\sim \Big(\frac{r}{\alpha^2\log r}\Big)^\alpha$,
i.e., up to a log-factor $\hat\gamma$ the rank approximation order of the solution exhibits still the same algebraic rate. Instead, for $\gamma(r)= e^{\alpha r}$ one obtains
$\hat\gamma(r)\sim e^{(\alpha r/C)^{1/3}}$ and thus a more significant relative degradation of low-rank approximability caused by $A^{-1}$.
Of course, when $f\in \Tcal_k(\HH^{t+\zeta})$ has a finite tamed rank $k$, it belongs to $\cA^\gamma(\HH^t,\HH^{t+\zeta})$ for
any growth function $\gamma$. This explains the high efficiency of low-rank approximations to the solution $u$ in such a case.

The result in Theorem \ref{thm:ts1} is in this form unusual since the approximants $u_r\in \Tcal_r(\HH^{t+2+\zeta})$ in \eqref{uappr} are not yet determined by finitely many parameters.
It remains to identify numerically viable {\em finitely parametrized} approximants $\bar u_r\in \Tcal_r(\HH^{t+2+\zeta})$ 
and bound the complexity of computing them.

We summarize next some related simplified but representative results from \cite{DDGS}. Note first that stable approximability of the data $f$
by somewhat smoother low-rank expressions in $\Tcal_r(\HH^{t+\zeta})$ can be viewed as {\em excess-regularity} of the data that, according
to Theorem \ref{thm:ts1}, is in this form inherited by the solution. Moreover, a rank-one tensor $g$ belongs to $\HH^s$ only if its factors $g_i$
belong to the corresponding component spaces $\HH^s_i$. Therefore, the excess-regularity $\zeta$ affects the
number of parameters needed to form an $\e$-accurate rank-$r$ approximant. 

More precisely, the following facts are somewhat specialized cases of the results in 
\cite{DDGS}. For convenience we assume here $1$-dimensional component domains $\Omega_i\subset \R$,  $\zeta \in (0,2]$, and $t=-1$. 
Then, for any $f\in \Acal^{\gamma}(\HH^{-1},\HH^{\zeta-1})$ and any $\e>0$ there exists  a $\bar u (\e) \in \Tcal_{\bar r(\e)}(\HH^{\zeta+1})$, which is determined by $N(\e,d)<\infty$ parameters, such that
\be
\label{appr2}
\|u-\bar u(\e)\|_1\le \e ,%
\quad \tripnorm{\bar u(\e)}_{\bar r(\e),1} 
\le C_1 \|f\|_{\Acal^{\gamma}(\HH^{-1},\HH^{\zeta-1})},
\ee
where
\be
\label{barr}
\bar  r(\e)\le  C(\zeta)\gamma^{-1}\big(\bar C\|f\|_{\Acal^{\gamma}}/\e\big)|\ln \e|^2,\quad N(\e,d) \le C_2 d\e^{-1/\zeta} \bar r(\e)^{1+\frac{1}{\zeta}}.
\ee
Since in view of \eqref{reps}, an $\e$-accurate approximation of $f$ requires at most a rank $r(\e)\lesssim \gamma^{-1}\big(
\|f\|_{\Acal^{\gamma}}/\e\big)$, the solution rank $\bar r(\e)$ is only mildly increased by a factor of the order 
of $|\ln \e|^2$. Moreover, the total number $N(\e,d)$ of
degrees of freedom exceeds an ``ideal count'' of $d\e^{-1/\zeta} \bar r(\e)$, obtained when  all tensor factors of $\bar u(\e)$ 
of smoothness $1+\zeta$ were known beforehand, by the factor $\bar r(\e)^{\frac{1}{\zeta}}$.

A pivotal constituent behind the above results is again approximation by {\em exponential sums}. Specifically, one can employ
a somewhat simpler version
of Theorem \ref{thm:expsumrel}, namely a classical result by Braess and Hackbusch \cite{BH1,BH2}  stating that
\be
\label{exps1}
\sup_{x\in [1,\infty)} \Big|\frac 1x - S_r(x)\Big| \le C e^{-\pi\sqrt{r}},\quad S_r(x)= \sum_{k=1}^r \omega_{r,k}e^{-\alpha_{r,k}x}.
\ee
Defining operator exponentials $e^{-A}v$ through the eigensystems \eqref{Aeigen}, one can prove the following result
that takes the scale-change caused by $A^{-1}$ into proper account.
\begin{theorem}
\label{thm:expappr}
There exists a constant $C_0$ such that for $t\le s\le 1$, $s-t\le 2$, one has
\be
\label{Ainv}
\|A^{-1} - S_r(A)\|_{ \mathcal{L}(\HH^t,\HH^s)} \le C_0 e^{-\frac{(2-s+t)\pi}{2}\sqrt{r}},\quad r\in \N.
\ee
\end{theorem}
Of course,  this allows one to determine $R(\e)$ such that for any given rank-one tensor $\tau = \tau_1\otimes\cdots\otimes \tau_d$
one can approximate with accuracy $\e$ in a desired smoothness scale $A^{-1}\tau$ by the rank-$R(\e)$ tensor
$$
S_{R(\e)}(A)\tau = \sum_{k=1}^{R(\e)} \omega_{R,k}\Big(\bigotimes_{j=1}^d e^{-\alpha_{R,k}A_j}\tau_j\Big), \quad \text{where}
\,\, R=R(\e)\sim |\ln\e|^2,
$$
explaining the rank-increase by a factor $|\ln\e|^2$.

The {\em representation complexity} bounds \eqref{barr} make use of the implied regularity of the tensor factors. They can 
be extended to bounds for the {\em computational complexity} of actually {\em }computing an accuracy controlled low-rank approximant
by devising a numerical scheme for evaluating  exponential sums of the form $S_{R(\e)}(A)\tau$. Of course, using the eigensystems 
of $A$ is not a numerically viable option. Instead,
 approximating
 exponential factors of the type $e^{-\alpha A_j\tau_j}$ in low dimensions can
be based on the Dunford integral representation
\be
\label{Dun}
e^{-\alpha A_j}\tau_j = \frac{1}{2\pi i} \int_\Gamma e^{-\alpha z} (z I - A_j)^{-1} \tau_j dz,
\ee
where $\Gamma\subset \mathbb{C}$ is a suitably chosen analytic curve separating the spectrum of $A_j$ from the left complex half plane,
and permitting an analytic extension of the integrand to a sufficiently wide region, see \cite{DDGS} for details. 
Sharp estimates for the mapping properties of the low-dimensional operator exponentials $e^{-\alpha A_j}$ and  using
exponential convergence of sinc-quadrature, one can  devise for any given target accuracy $\e$ appropriate
discretizations of the resolvent problems $(zI-A_j)u_j(z) = \tau_j$ for suitably chosen quadrature points $z$. Thus, in contrast to
earlier work in \cite{GHK}, using such integral representations, the order of first discretizing the operator $A_j$ and then
using quadrature, is here reversed in order to exploit mapping properties for error controlled computation. When the component operators
$A_j$ are indeed second order $H^1_0(\Omega_j)$-elliptic and if $\HH^s$ agrees with $H^1_0(\Omega_j)\cap H^s(\Omega_j)$, $s\in [1,2]$,
one can  bound the {\em number of operations and used degrees of freedom} for computing $\bar u(\e)$ by
\be
\label{compl}
d^{1+\bar \rho/\zeta} \e^{-\bar\rho/\zeta} \gamma^{-1}(C\|f\|_{\Acal^\gamma}/\e)|\ln \e|^2\Big(\ln \Big(\frac{d}{\e}\Big)\Big)^2,
\ee
where $\bar\rho$ is a fixed positive number. This is slightly more pessimistic than \eqref{barr} but nevertheless
shows that the curse of dimensionality is clearly avoided for a wide scope of data $f$.

Exponential sum based schemes are very performant but restricted to the particular operator type \eqref{Delta}.
Nevertheless, in particular, sub-exponential rates are seen to provide relevant benchmarks for high-dimensional
diffusion problems with low-rank data. The performance of the more general iterative approach, outlined in Section \ref{sec:solstrat},
will therefore be discussed  in Section \ref{sec:complexity} for  such benchmarks.

\subsection{Multi-Parametric Problems}\label{ssec:many}

As mentioned earlier in Section \ref{sec:solstrat}, the uniform ellipticity assumption on the parametric 
family $a(y), y\in \pdom$, ensures ellipticity of \eqref{parametric} in the tensor product Hilbert space
\be
\label{paraH}
V :=  H^1_0(\Omega)
 \otimes L_2(\pdom,\mu) =  H^1_0(\Omega)
 \otimes\Big(\bigotimes_{j=1}^d L_2\bigl((-1,1), \textstyle\frac {d y_j}2\displaystyle)\Big) = L_2(\pdom,H^1_0(\Omega),\mu).
\ee
The mapping
$u: y\mapsto u(y)\in \cH$
is then well-defined and the set $u(\pdom) \subset \cH$ of all states that can be attained when traversing the parameter domain, sometimes referred to as {\em solution manifold}, is compact. Hence the {\em Kolmogorov $n$-widths}
$$
d_n(u(\pdom))_V:= \inf_{{\rm dim}(W)=n}\sup_{y\in \pdom}\inf_{w\in W}\|u(y)-w\|_V,
$$
tend to zero. They quantify linear approximability of $u(\pdom)$ in the worst case sense with respect to $\pdom$ viewing $u$ as a mapping into the smaller Bochner space $L_\infty(\pdom,H^1_0(\Omega))$. A natural accuracy measure for low-rank approximation in $V$ is 
the mean-squared error
\be
\label{deltan}
\delta_n(u,\mu)^2_V := \inf_{{\rm dim}(W)=n} \int_\pdom \min_{w\in W}\|u(y)-w\|_{H^1(\Omega)}^2d\mu(y).
\ee

As exploited earlier, best subspaces realizing $\delta_n(u,\mu)_V$ are obtained with the aid of the Hilbert-Schmidt operator
$M_u: L_2(\pdom,\mu)\to H^1_0(\Omega)$ defined by
$$
M_u \varphi := \int_\pdom u(y)\varphi(y)d\mu(y),
$$
through its Hilbert-Schmidt decomposition 
\be
\label{HSy}
M_u = \sum_{k=1}^\infty \sigma_k v_k \langle \cdot,\phi_k\rangle_{L_2(\pdom,\mu)},
\ee
where $\sigma= (\sigma_k)_{k\ge1}\in \ell_2(\N)$ is non-negative and non-increasing. Moreover,  the $(v_k)_{k\ge 1}, (\phi_k)_{k\ge 1}$
are orthonormal systems in $H^1_0(\Omega)$ and $L_2(\pdom,\mu)$, respectively. In particular, the spaces $W_n:= {\rm span}\,\{v_k:k=1,\ldots,n\}$
realize $\delta_n(u,\mu)_V$ and 
\be
\label{errn}
\delta_n(u,\mu)^2_V= \sum_{k>n}\sigma_k^2.
\eeqn
The largest $r$ for which $\sigma_r>0$ is the rank of $M_u$.

Since $\mu$ is a probability measure $\delta_n(u,\mu)_V$ is always dominated by the $n$-widths
\be
\label{Knw}
\delta_n(u,\mu)_V\le d_n(u(\pdom))_V, %
\ee
so that bounds for $d_n$ provide such for $\delta_n$ as well. In this context a crucial fact is {\em holomorphy} of the map $u$, see 
\cite{cds1,cds2,CDacta:15}, even when $d=\infty$, which is, in particular, responsible for decay rates of $\delta_n, d_n$ that are robust in $d$.
 
Quantifying such rates, however, requires  distinguishing 
two fundamental regimes, namely the case (a) $d<\infty$ possibly very large but fixed, and (b) $d=\infty$. In (a) 
all parametric components $y_i$ will be viewed as equally important. In scenario (b) uniform ellipticity necessitates 
a  decay of the parametric expansion functions $\psi_j$ in $L_\infty(\Omega)$.

\subsubsection{Finitely many parameters}
 It is straightforward to show (for instance, using Taylor expansion as in \cite{BC:15}) that for fixed $d<\infty$, the $n$-widths of the solution manifold $u(\pdom)$ satisfy
\be
\label{adecay}
d_n(u(\pdom))_\cH \le C e^{-cn^{1/d}},
\ee
where the constants $c,C$ depends on  $\bar a$ and $r$ in 
 the uniform ellipticity condition $r\le \bar a -\sum_{i=1}^d\abs{\psi_i} \le a(y)= \bar a + \sum_{i=1}^d y_i\psi_i$.

The rate \eqref{adecay} can be slightly improved to $e^{-\abs{\ln \rho}n^{\frac{1}{d-1}}}$ under the assumption
\be
\label{assA}
\sum_{i=1}^d A_i= \theta \bar A \quad \text{for a $\theta \in (0,1)$,}
\ee
for $A_i,\bar A$ as above, which holds when $\sum_{i=1}^d\psi_i =\theta \bar a$.
More precisely, it is shown in \cite{BC:15} that the partial sums $u_k(y)$ in a Neumann series expansion of $u(y)$ can be represented as
\be
\label{ps}
u_k(y) = \sum_{j=0}^{n(d-1,k)}\phi_{k,j}(y)\, v_j,
\ee
for some  $d$-variate polynomials $\phi_{k,j}$ and $v_l\in V$, and therefore have at most rank $n(d-1,k)\le (k+1)^{d-1}$.

Estimates of the type \eqref{adecay} give a first justification for the use of  sub-exponential decay rates  as benchmark rates. %
On the other hand, the obtained rates weaken substantially with increasing $d$.  
This has led to a careful study of more specialized types of parametric expansions to better understand whether and 
to what extent  low-rank approximation may substantially outperform 
approximations based on an \emph{a priori} fixed expansion system such as tensor product Taylor or Legendre expansions.

An important model case are 
 parametric expansions with {\em piecewise constant} coefficients of the form
 \begin{equation} \label{inclusions}
    a(y) = \bar{a} + \sum_{j=1}^d y_j \theta_j  ,
 \end{equation}
where $d<\infty$, $\bar a=1$,  $\theta_j= b_j \Chi_{D_j}$, $b_j\in (0,1)$, and the subdomains    $D_j\subset \Omega$ 
form a partition of $\Omega$.
 
 Specifically, consider first the spatially one-dimensional case $\Omega = (0,1)$.
  As shown in \cite{BC:15,BCD15}, for any $f\in \cH'$, $u(y)$ has then finite rank at most $2d-1$.
 In fact, $u(y)|_{D_i}$ must be spanned for each $y$  by $\Chi_I,x\Chi_I,F_i\Chi_I$ where $F_i''=f$ on $D_i$
 and $F_i$ vanishes at the end points of $D_i$. Boundary conditions and $d-1$ continuity conditions at the interior breakpoints
 leave $2d-1$ degrees of freedom so that $u(y)$ is contained in a linear space of dimension at most $2d-1$. Since therefore 
 $d_{2d-1}(u(\pdom))_V=0$, $u(y)$ has at most rank $2d-1$.

For higher spatial dimensions, $M_u$ is finite only in exceptional cases. 
For piecewise constant diffusion coefficients of the form \eqref{inclusions} with $\theta_j= \bar a \theta \chi_{D_j}$, significantly refined information is derived in \cite{BC:15}
by decomposing $H^1_0(\Omega)$ into $\bigoplus_{i=1}^d H^1_0(D_i)$ and its orthogonal complement $W$ with respect to
$\langle\cdot,\cdot\rangle_{\bar A}$. Employing corresponding $\bar A$-orthogonal projections, one can write 
\[
u(y)= u_W(y) + \sum_{i=1}^d u_i(y),
\] 
where  the $u_i(y)$ are solutions to
$$
(1+\theta y_i) \int_{D_i} \bar a \nabla u_i(y)\cdot\nabla v\, dx = \int_{D_i}f\,v\, dx, \,\, v\in H^1_0(D_i),\quad u(y)|_{\Omega\setminus D_i}=0.
$$
Since the $u_i$ have rank one depending only on $y_i$ the portion $\sum_{i=1}^d u_i(y)$ has at most rank $d$. 
To determine the rank of $u_W$, it is written as the harmonic extension of a skeleton component $u_\Gamma$.
It is then shown that $u_\Gamma(y)$ has a Neumann series representation which ultimately requires estimating the ranks
of its partial sums $u_{\Gamma,k}(y)$. 

In the case $d=4$ and $D_i$ being congruent squares, this analysis is carried out
showing that in this case the ranks of the partial sums $u_k(y)$ of the Taylor expansion of $u(y)$ are bounded by $8k+5$, which then 
leads to a bound $d_n(u(\pdom))_V \le C e^{-cn}$.  Hence,
 for each $n\in\N$ one can find functions $u^\rs_k\in H^1_0(\Omega)$, $u^\rp_k \in L_2(Y,\mu)$ such that
 \begin{equation}\label{pwconstexponential}
   \Bignorm{ u - \sum_{k=1}^n u^\rs_k \otimes u^\rp_k}_{L_2(\pdom,H^1_0(\Omega))} \lesssim e^{- c n}.
 \end{equation}
 Remarkably, numerical experiments reveal that for partitions into four non-congruent quadrilaterals, corresponding singular values
 decay at a significantly slower, yet still subexponential rate.
 
Moreover, numerical experiments support a similar behavior for larger $d$, with $c$ in the estimate \eqref{pwconstexponential} depending weakly (approximately as $1/d$) on $d$.
 Very similar results are obtained for hierarchical tensor approximations, where also the intermediate ranks for 
 matricizations with respect to the parametric variables enter. While to our knowledge, there are no theoretical bounds on
 these ranks, the numerical results indicate that the ranks for the spatial matricizations 
 dominate and behave essentially as in the previous case where the parametric variables are not separated.
 
\subsubsection{Infinitely many parameters}\label{sssec:infparam}
The case $d=\infty$, where we consider diffusion coefficients $a(y) = \bar a + \sum_{i=1}^\infty y_i \psi_i$ parameterized by $y \in Y = (-1,1)^\N$, turns out to be significantly different. A tensor product basis for $L_2(Y,\mu)$ is provided by the product Legendre polynomials $L_\nu(y) = \prod_{i=1}^\infty L_{\nu_j} (y_j)$ for $\nu \in \Fcal := \{ \nu \in \N_0^\N \colon \nu_j = 0 \text{ for almost all $j\in \N$}\}$, giving rise to the orthonormal polynomial basis expansion 
\begin{equation}\label{eq:legendreexp}
u(y)= \sum_{\nu\in \cF} u_\nu L_\nu(y), \quad u_\nu = \int_Y u(y)\,L_\nu(y)\,d\mu(y).
\end{equation}
Also here, a natural first question concerns the separation of spatial and parametric variables as in \eqref{pwconstexponential}, that is, the approximation of $u$ by functions of the form $\sum_{k=1}^n u^\rs_k \otimes u^\rp_k$. The smallest possible error $\delta_n(u,\mu)_V = (\sum_{k>n} \sigma_k^2 )^{1/2}$ is attained by the truncated Hilbert-Schmidt decomposition of $M_u$ and thus determined by the decay of the singular values $(\sigma_k)_{k\in\N}$ of $M_u$.

Truncations of the expansion \eqref{eq:legendreexp} provide alternative low-rank approximations, where one takes $u^\rs_k := u_{\nu_k}$, $u^\rp_k := L_{\nu_k}$ with $(\nu_k)_{k\in\N}$ chosen such that $\norm{u_{\nu_k}}_{H^1(\Omega)} \geq \norm{u_{\nu_{k+1}}}_{H^1(\Omega)}$, which yields an approximation error $(\sum_{k>n} \norm{u_{\nu_k}}_{H^1(\Omega)}^2 )^{1/2}$ with $n$ terms. By the best approximation property of the Hilbert-Schmidt decomposition, clearly $\sum_{k>n}\norm{u_{\nu_k}}^2_{H^1(\Omega)} \geq \sum_{k> n}\sigma_k^2$. 

The decay of $\norm{u_{\nu_k}}_{H^1(\Omega)}$ as $k\to \infty$, which depends mainly on the functions $\psi_j$ and their decay as $j\to\infty$, is well studied, see e.g.\ \cite{cds1,cds2,CDacta:15,BCM:2015}. One obtains results of the type $( \norm{u_\nu}_{H^1(\Omega)} )_{\nu \in \cF} \in \ell_p(\cF)$ for some $p \in (0,2)$, which implies $(\sum_{k>n} \norm{u_{\nu_k}}_{H^1(\Omega)}^2 )^{1/2} \lesssim n^{-1/p+1/2}$. There are only very limited results on the decay of $\sigma_k$ in this setting; it depends also, e.g., on the right hand side $f$. 

The following result on $(\sigma_k)_{k\in\N}$ was obtained in \cite[\S6]{BCD15}, adapting arguments from \cite[\S4.1]{BCM:2015}:
For $\Omega=(0,1)$ and a certain class of $\psi_j$ such that $u(y)$ is easy to handle analytically, there always exists $f \in H^{-1}(\Omega)$ 
such that $( \norm{u_\nu}_{H^1(\Omega)} )_{\nu \in \cF} \in \ell_p(\cF)$ and $(\sigma_k)_{k\in\N}\notin \ell_{p'}(\N)$ for any $p'< p$. In other words, with appropriate $f$, the singular values $\sigma_k$ have the \emph{same asymptotic decay} as the decreasingly ordered Legendre coefficient norms $\norm{u_{\nu_k}}_{H^1(\Omega)}$. This also yields lower bounds on the Kolmogorov widths $d_n(u(\pdom))_{H^1(\Omega)} \geq  \delta_n(u,\mu)_{H^1(\Omega)} \gtrsim n^{-1/p+1/2}$ in these settings. Thus, the performance concerning the number of terms $n$ attainable by any low-rank approximation is, for an adverse choice 
of data $f$,  asymptotically no better than that of the simpler Legendre expansion \eqref{eq:legendreexp}.

As numerical experiments in \cite{BCD15} show, this situation is also observed in generic examples with typical choices of $\psi_j$ and $f$, where this has not been proved analytically. In summary, in such problems with infinitely many parameters, one can therefore only expect limited gains from optimized low-rank approximations in comparison to Legendre expansions \eqref{eq:legendreexp}. Indeed, as explained in more detail in Section \ref{sec:complexity}, corresponding SVD-based computational methods are generally more expensive, since they scale nonlinearly with respect to the expansion ranks.

\section{Adaptive Error Control and Computational Complexity}\label{sec:complexity}

 A central objective in adaptive methods for high-dimensional problems is the reliable quantification of errors. 
For scenario (I), one approach to determine rigorous error bounds has been recently proposed in \cite{Dol:19} employing the {\em Hypercircle Principle} 
in combination with flux-approximations in tensor-train format. However, lower bounds do not seem to be available in this format.
Hence, while the involved constant is just one, a quantification of a possible overestimation and its dependence on $d$ seems to remain
unclear.  Moreover, the approach does not offer any provision for updating the 
given discretization based on the computed bounds.
When estimating instead residuals directly,  as suggested in this article,
one issue is the quality of error estimates in that one needs to ensure that all involved constants are controlled and the resulting bounds remain meaningful for large dimension $d$. For instance, in naive generalizations of lower-dimensional error estimators, constants in the estimates may depend exponentially on $d$.
 A second issue is the numerical cost of their evaluation, which should equally remain under control for large $d$.
 
The approach that we follow here addresses both issues in conjunction with strategies for updating the approximations. It is based on the proportionality between errors and residuals \eqref{res} and approximation of these infinite-dimensional residuals. To ensure the quality of error estimates for large $d$, we rely in particular on expressing the underlying problem in terms of suitable basis functions. Achieving near-optimal computational cost of the residual approximation prohibits, of course,  ever accessing the underlying full coefficient tensor but requires a somewhat specialized treatment in each of the problem scenarios that we consider.

The resulting methods for error estimation can, in principle, be used in various different constructions of adaptive solvers. In Section \ref{ssec:complexity}, we combine them with a scheme for which we can subsequently show near-optimal asymptotic computational costs.
 
\subsection{Residual Approximations}\label{sec:residualapprox}
The accuracy-controlled approximation of residuals $\bA \bv -\bbf$ for given finitely represented $\bv$, needed in any realization of \eqref{itertemplate}, requires some means of approximating $\bbf$ in the desired form, and more crucially a scheme for approximately applying the infinite matrix $\bA$.
Devising such a routine with provably optimal performance, one faces different obstructions in the two problem scenarios (I), (II). 

\subsubsection*{{\rm(I)} Diffusion Equations} Recall that an important requirement on a Riesz basis $\Phi$ for the energy space
$V$ is that it is a rescaled version of an orthonormal basis $\Psi$ for $L_2(\Omega)$; the latter property is crucial for avoiding constants in the estimates that deteriorate exponentially in $d$. In this case the (unbounded)  representation 
$\bT := S_\Psi' AS_\Psi$ preserves the (formal) low-rank structure of the diffusion operator while the scaled version 
$  \bA = S_{\tilde\Phi}' A S_{\tilde\Phi} = \tilde \bS^{-1}_\infty \bT \tilde \bS^{-1}_\infty$ from \eqref{A1} has infinite
rank. A key component of a computable residual approximation in low-rank format thus is to replace the ideal scaling matrix $ \tilde \bS^{-1}_\infty$ by an approximation
$ \tilde \bS^{-1}_m$. Here, $m = m(\bv,\eta)$ depends through \eqref{m} on the support of the vector $\bv$,  which $\bA$ is  to be applied to,  
as well as on the target accuracy $\eta$. This ensures that the  matrix $\bA_{m(\bv,\eta)}:=  \tilde \bS^{-1}_m \bT \tilde \bS^{-1}_m$
has now finite rank while, on account of Theorem \ref{thm:expsumrel}, the restriction to the activated basis elements is still well-conditioned and thus convergence of \eqref{itertemplate} is still ensured.

The matrix $\bA_{m(\bv,\eta)}$ is still infinite and its approximate
application   makes essential use of the {\em near-sparsity} or 
{\em compressibility} of  properly rescaled versions of the low-dimensional blocks $\bT^{(i)}_{n_i}$ of $\bT$. 
Exploiting compressibility of such scaled blocks, in turn, is facilitated by the  near-separability of the low-rank   
scaling matrices. 
 
 To indicate how this can be used to approximate the infinite-dimensional residual $\bA \bu^n-\bbf$, whose $\ell_2$-norm is uniformly proportional to the error $\|\bu-\bu^n\|_{\ell_2(\cI)}$, within a suitably updated 
 dynamic accuracy tolerance $\eta_n$, we recall the notion of {\em $s^*$-compressibility} which plays a 
 pivotal role in adaptive wavelet methods, \cite{Cohen:01,Cohen:02}. 
 An infinite  matrix $\mathbf{B}\in \R^{\vee\times\vee}$
  is called
\emph{$s^*$-compressible} if for any $0 < s < s^*$, there exist 
summable positive sequences $(\alpha_j)_{j\geq 0}$, $(\beta_j)_{j\geq 0}$ 
and for each $j\geq 0$, there exists $\mathbf{B}_j$ with at most $\alpha_j 2^j$
nonzero entries per row and column, such that $\norm{\mathbf{B} - \mathbf{B}_j}
\leq \beta_j 2^{-s j}$. 

To approximate $\mathbf{B}\bv$ for a given $\bv\in \ell_2(\vee)$ one combines {\em a priori} knowledge about the matrix $\mathbf{B}$
with {\em a posteriori} information on the input $\bv$, e.g., by partitioning the support of $\bv$ into disjoint subsets $\Lambda_j$
corresponding to the magnitude of its entries in decreasing order. We then define a $\bv$-dependent approximation $\tilde\bB_J$ of $\bB$,
\begin{equation}
\label{motivate}
  \bB \bv \approx \tilde\bB_J \bv := \sum_{j=0}^{J} \bB_{J-j}(\bv|_{\Lambda_j})	,
\end{equation}
which satisfies
\begin{align*}
\norm{ \bB\bv - \tilde\bB_J \bv }  &=  
\Bignorm{ \sum_{j=0}^{J}(\bB- \bB_{J-j})(\bv|_{\Lambda_j}) + \bB (\bv|_{\Lambda_{J+1}}) } \\
  & \leq \sum_{j=0}^{J} \norm{\bB- \bB_{J-j}} \norm{\bv|_{\Lambda_j}} + \norm{\bB} \norm{ \bv|_{\Lambda_{J+1}} } .  
\end{align*}
Note that for $j$ close to $J$, $\bB_{J-j}$ is a coarse approximation, whereas the 
most accurate approximations of $\bB$ are applied to the large contributions in $\bv$. 

For high-dimensional input tensors $\bv \in \ell_2(\vee^d)$, the crucial step is then to apply approximations of each lower-dimensional component in the tensor decomposition of $\mathbf{T}$ separately in each tensor mode, with accuracies controlled by subdividing the sequences $\pi^{(i)}(\bv)$ for each mode $i=1,\ldots,d$.
As shown in \cite{BD2}, one can then combine this with the approximate rescaling in a routine $\apply(\bv;\eta): \bv \to \bw_\eta$
with
\beqn
\label{weta-def}
\bw_\eta := \tbS^{-1}_{m(\eta;\bv)}\tbT_{J(\eta)}\tbS^{-1}_{m(\eta;\bv)}\bv,
\eeqn
for exponential sum truncation ranks $m(\eta;\bv)$ chosen according to \eqref{m} and compression depth $J(\eta)$,
such that $\norm{\bA\bv -\bw_\eta} \leq \eta$. 

Instead of the energy norm, one can instead measure accuracy in a weaker norm
such as the $L_2(\Omega)$-norm. Since this is a cross-norm one would intuitively expect achieving desired accuracy 
thresholds at the expense of lower ranks. This is indeed the case as shown in \cite{BD3}, where rigorous error control in 
$L_2(\Omega)$ is established again in the framework of an equivalent problem formulation in sequence space.
The relevant mapping properties are now realized through an {\em unsymmetric} preconditioning, which requires only a single scaling operation in each application of $\bA$ of the form
\be
\label{A2n}
\bA_{m(\bv,\eta)} = \bS_{\Phi_{m(\bv,\eta)}}^{-2} \bT  ,
\ee
see \cite{BD3} for further details.
 
\subsubsection*{{\rm(II)} Parametric PDEs} As indicated at the beginning of Section \ref{ssec:core}, one can tensorize 
any Riesz-basis $\Phi$ of $H^1_0(\Omega)$ with an orthonormal basis of $L_2(\pdom,\mu)$. A natural choice for this orthonormal basis are the product Legendre polynomials as in Section \ref{sssec:infparam}, with an orthonormal basis given by $\{ L_\nu\}_{\nu \in \mathcal{G}}$, where $\mathcal{G} = \N_0^d$ for $d<\infty$ and by $\mathcal{G} = \cF$ for $d=\infty$. The resulting
system $\Phi:=\{\varphi_\lambda \otimes L_\nu\}_{(\lambda,\nu)\in \sidx \times \mathcal{G}}$ is then a Riesz basis of $L_2(\pdom,H^1_0(\Omega))$,
which we now use to transform \eqref{parametric} into an  equivalent system over $\ell_2(\sidx\times \mathcal{G})$.
Defining then with $\psi_0 := \bar a$ for $j\in \N_0$
\be
\bA_j := \bigl( \langle \psi_j \nabla  \varphi_{\lambda'}, \nabla \varphi_{\lambda}\rangle \bigr)_{\lambda,\lambda' \in \sidx},\quad 
\bbf := \big(\langle f, \varphi_\lambda\otimes L_\nu\rangle\big)_{(\lambda,\nu)\in \sidx\times\mathcal{G}},
\ee
as well as 
\be
 \bM_j := \left( \int_{\pdom} y_j  L_\nu(y) L_{\nu'}(y) \,d\mu(y)  \right)_{\nu,\nu'\in\mathcal{G}},\quad j\geq 1,
\ee
and $\bM_0$ as the identity on $\ell_2(\mathcal{G})$,
we obtain an equivalent problem
\beqn
\label{equiv}
\bA \bu = \bbf,\quad \text{where}\quad  \bA := \sum_{j=0}^d \bA_j \otimes \bM_j  ,
\eeqn
on $\ell_2(\sidx \times \mathcal{G})$. Due to the classical three-term recurrence relations for Legendre polynomials, the matrices $\bM_j$ are
{\em bi-diagonal} for $j\geq 1$.

In the case of fixed finite dimension $d<\infty$, the hierarchical tensor format is a natural choice. In fact, it poses significantly less
difficulties than in the case of high-dimensional diffusion equations, since the mapping $S_\Phi$ is a Kronecker product and the sparse matrices $\bM_j$
need no additional compression. 
The matrices $\bA_j$ can be approximated similarly to \eqref{motivate}, with accuracies controlled by norms of segments of $\norm{\pi^{(\rs)}(\bv)}$ of the spatial contractions
\[
\pi^{(\rs)}_\lambda(\bv) :=   \Big(\sum_{\nu\in \mathcal{G}}|\bv_{\lambda,\nu}|^2\Big)^{1/2}, \quad \lambda \in \sidx,
\]
for inputs $\bv \in \ell_2(\sidx\times \mathcal{G})$, which can be computed from the ${H}$SVD as in \eqref{eq:contreval}.
As shown in the next section, corresponding complexity estimates are slightly more favorable than for scenario (I).

In problems with $d=\infty$, residual approximations are more complex due to the additional requirement of \emph{adaptive truncation of dimensionality:}
to ensure efficiency with controlled errors, the number of parameters $y_j$ that are activated needs to be adjusted to discretization and low-rank approximation errors. This can be controlled adaptively, taking into account the known decay properties of the parameter expansion functions $\|\psi_j\|_{L_\infty(\Omega)}$, via estimates for the operator truncation error of the type
\be
\label{eJ}
 e_{J} := \Big\| \sum_{j > J } \bA_j \otimes \bM_j \Big\|\le C J^{ -S},
 \ee
 with some fixed $S>0$. 
 Such an adaptation has been considered in \cite{Gittelson:13,Gittelson:14} for generic $\psi_j$, but as observed in \cite{BCD15}, the value of $S$ is strongly influenced by structural features of this expansion. A particularly favorable case are expansions with a multilevel structure, where each $j$ is uniquely associated to a pair $(\ell(j), k(j))$ with inverse mapping $j(\ell,k)$, such that there are $\mathcal{O}(2^{m\ell})$ functions on level $\ell$ and
\be
\label{ml}
  \sum_{k} \abs{\psi_{j(\ell, k)}} \leq C 2^{-\alpha  \ell}, \quad \ell \geq 0,
\ee
for some $\alpha>0$, with a uniform constant $C>0$.
Here $\ell(j) \in \N_0$ describe the scale and $k(j) \in \Z^m$ (with $\Omega \subset \R^m$) the spatial localization of $\psi_j$.
Such a condition is satisfied in particular when ${\rm diam}({\rm supp}(\psi_j))\sim 2^{-\ell(j)}$, there exists $M>0$ such that for each $\ell, k$, the number $\#\{ k' \colon \supp \psi_{j(\ell,k)}\cap \supp \psi_{j(\ell,k')}\neq \emptyset \} \leq M$, and $\norm{\psi_j}_{L_\infty(\Omega)} \sim c_j \sim 2^{-\alpha \ell(j)}$.
Expansions with the property \eqref{ml} can be constructed for wide classes of random fields, see \cite{BCM:2016}, where one obtains $S = \alpha/m$.
 
 In principle, for any fixed set of modes $x, y_1,\ldots,y_J$, one can apply the hierarchical format with $J+1$ tensor modes, which leads to some substantial additional difficulties in the complexity analysis due to the changing dimensionality. We therefore restrict the following discussion to the basic case of  separating spatial indices in $\sidx$ and parametric indices in $\cF$ as in Section \ref{sssec:infparam}, that is, to estimating errors in low-rank approximations to $\bu$ of the form 
   \be
\label{uelr}
\bar\bu =   \sum_{k=1}^{r}\bar \bU^\rs_k \otimes \bar\bU^\rp_k.
\ee
with $\bU^\rs_k$ and $\bar\bU^\rp_k$ finitely supported on $\sidx$ and $\cF$, respectively.
In view of the results outlined in Section \ref{sssec:infparam}, direct Legendre expansion can be advantageous compared with \eqref{uelr}
regarding asymptotic costs. Nevertheless, despite the fact that the number of terms $r$ is possibly not significantly smaller
than for sparse Legendre expansions, and the total computational cost may in fact be higher, an efficient computation of low-rank approximations with controlled error in $L_2(\pdom,H^1_0(\Omega))$ and near-optimal basis functions as in \eqref{uelr} can still be of interest in many application scenarios, for instance in the case of many parameter queries.

In the case of a low-rank approximation of the residual with simultaneous sparse approximation in the tensor modes, the decomposition of $\bv$ is more
involved. For $p,q = 0,1,2,\ldots$ an input $\bv$  is then decomposed further into segments $\bv_{[p,q]}$ of rank $2^p$, where $p$ refers to consecutive groups of $2^p$ singular values of $\bv$, and $q$ refers to the best $2^q$-term approximations 
of the parametric contractions $\pi^{(\rp)}(\bv)$, defined by  
\[
  \pi^{(\rp)}_\nu(\bv) :=   \Big(\sum_{\lambda\in \sidx}|\bv_{\lambda,\nu}|^2\Big)^{1/2}.
\] 
The routine $\apply(\bA,\bv;\eta) \to \bw_\eta$ then yields finitely supported low-rank approximations 
\begin{equation}\label{eq:bweta}
 {\bw_{\eta} := \sum_{p,q\ge 0}\sum_{j=1}^{J_{p,q}({\eta})}({\tilde\bA_{j,p,q}} \otimes \bM_j)\bv_{[p,q]}} \,,
\end{equation}
with ${ \|\bw_{ \eta} -\bA\bv\|\le { \eta}}$,
where the number of terms $J_{p,q}(\eta)$ is determined in an a posteriori fashion from the 
norms $\|\bv_{[p,q]}\|$ as well as known bounds for the truncation errors $e_{J_{p,q}}$ from \eqref{eJ}. Moreover, 
$(\tilde\bA_{j,p,q} \otimes \bM_j)\bv_{[p,q]}$ is an error-controlled approximation of $(\bA_j \otimes \bM_j) \bv_{[p,q]}$ obtained as in \eqref{motivate}.
Here, the change in compressibility of $\bA_j$ as $j\to \infty$, which typically  is driven by increasingly oscillatory features of the functions $\psi_j$, crucially determines the resulting computational costs. Also in this respect, a multilevel structure of the parametric expansion, as in \eqref{ml},  turns out to be highly advantageous, in that it also leads to improved compressibility of the matrices $\bA_j$. In summary, under this condition one can construct the approximation \eqref{eq:bweta} at near-optimal computational costs.
For further details, we refer to \cite{BCD15}.

\begin{remark}
	A different technique for residual approximation for such problems, using finite element discretizations in space but also relying on contractions $\pi^{(i)}(\bv)$ similarly to the present approach, was proposed in \cite{EPS17} and shown to provide an upper bound for the error. However, for this method no suitable complexity estimates are available.
\end{remark}

\subsection{Computational Complexity of Solvers}\label{ssec:complexity}

Independently of any specific solution algorithm, the residual approximations, discussed above, provide information on numerical errors in high-dimensional problems, where the costs of evaluating these approximations in terms of their quality depend only mildly on the dimensionality of the problem.
Such residual approximations can, in particular,  be  naturally used  in iterative solvers as discussed in Section \ref{sec:iter}, which gradually refine simultaneous sparse and low-rank approximations.

As indicated earlier, the residual approximations, as well as many further operations required in iterative solvers, increase the representation complexity of iterates, especially concerning their tensor representation ranks. The computational complexity of the resulting methods is therefore a crucial issue. Ideally, it should be directly related to the convergence of corresponding best approximations; in the case of low-rank approximations, however, one can generally not expect computational costs to remain strictly proportional to the number of degrees of freedom, since the costs of many required procedures such as SVD scale nonlinearly in the ranks.

A basic strategy for controlling approximation complexity, following the template \eqref{itertemplate}, is to combine error reduction (based, in particular, on residual approximations) with a suitable complexity reduction. Whereas maintaining convergence to the exact solution is straightforward with the ingredients we have introduced so far, relating the computational complexity to the quality of best approximations is substantially more delicate. 
We now discuss two basic approaches that rely on hard and soft thresholding with judiciously chosen thresholding parameters. 

\subsubsection{Hard thresholding}
This first approach is based on the observation that if an input $v$ is known to be quantifiably close to an unknown object $u$ then, using properly chosen thresholds, these routines produce a new approximation with a slightly larger error but with near-minimal ranks or representation complexities, respectively. 

We will first frame the basic mechanisms underlying the reduction procedures in abstract terms. 
To this end, let ${H}$ be a Hilbert space and let $\{ \mathcal{S}_r \}_{r\in \N}$ be a family of subsets from which approximations are selected, where the parameter $r$ controls their complexity.
For $v \in {H}$ and $\eta >0$, let 
\[ 
 \bar{r}(v,\eta) := \min\{ r \colon (\exists \,w \in \mathcal{S}_r\colon \norm{v - w} \leq \eta) \}, 
 \]
 which is the minimal complexity parameter for which the best approximation from these sets attains accuracy $\eta$.
  In addition, assume that there exists a linear projection $\bar{P}(v,\eta)$ onto $\mathcal{S}_{\bar{r}(v,\eta)}$ such that $\norm{v - \bar{P}(v,\eta) v} \leq \eta$.

\begin{lemma}
\label{lem:coars}
For any $r$, let $\hat{P}_r \colon \cH \to \mathcal{S}_r $ and $\hat{e}_r\colon \cH\to \R^+$ be such that 
\[ 
 \norm{v - \hat{P}_r(v)} \leq \hat{e}_r(v),
 \] where for some $\kappa>0$, 
 \[\hat{e}_r(v) \leq \kappa \inf_{w \in \mathcal{S}_r} \norm{v - w}.
 \]  
 For any $\alpha, \eta >0$, and for any $u, v \in \cH$ with $\norm{u - v} \leq \eta$, let $\hat r$ be minimal such that $e_{\hat r} (\bv) \leq \kappa(1+\alpha)\eta$.
Then, for $\hat v := \hat{P}_{\hat r}(v) \in \mathcal{S}_{\hat r}$  one has  
\be
\label{rclaim}
\norm{u-\hat v}\leq ( 1 + \kappa(1+\alpha))\eta\quad  \text{while} \quad \hat r \leq \bar{r}(u, \alpha\eta).
\ee
\end{lemma}
\begin{proof} Given $u,v, \eta>0$ as above, by linearity of $\bar P$, we have for any fixed $\alpha >0$
\begin{align}
\label{step1}
\|v - \bar{P}(u, \alpha\eta) v \|_\cH &\le \|(v-u)- \bar{P}(u,\alpha\eta) (v-u )\|_\cH\nonumber\\
&\qquad  + \|u- \bar{P}(u, \alpha\eta) u\|_\cH \le (1+\alpha)\eta.
\end{align}
On the other hand, we have for any $r$ that
\[
\|v - \hat P_r(v)\|_\cH \le \kappa \inf_{w\in \mathcal{S}_r} \|v- w\|_\cH,
\]
so that by \eqref{step1}, we have for $r=\bar r(v,\alpha \eta)$
\[
\|v - \hat P_{\bar r(v,\alpha\eta)} (v)\|_\cH \le \kappa\inf_{w\in \Scal_{\bar r(v,\alpha\eta)}}\|v-w\|_\cH
\le \kappa(1+\alpha)\eta.
\]
Since, by definition $\hat r$ is the minimal index providing accuracy $\kappa(1+\alpha)\eta$,  we conclude that
$\hat r\le \bar r(v,\alpha\eta)$ so that the second relation in \eqref{rclaim}
follows, while the first relation holds by the triangle inequality.
\end{proof}

Lemma \ref{lem:coars} means roughly the following: suppose one has constructed an explicit finitely parametrized approximation 
$v$ to an unknown $u$  for which an error bound $\eta$ is known. In the present context $v$ is the result of an error-controlled  reduction step
$\Fcal_n$ in \eqref{itertemplate}. Then, {\em re-approximating} the explicitly given $v$ within a judiciously chosen tolerance, which is
somewhat larger than the error bound $\eta$, provides a new approximation $\hat v$ to the unknown $u$ with a slightly worsened 
accuracy but with near-optimal complexity.

More precisely, the approximation complexity of the computed recompressed approximation $\hat v$ in \eqref{rclaim} can be quantified in terms of the corresponding best approximations of $u$. For instance, take $\Sigma_r$ as defined in \eqref{eq:sigmadef}, with a fixed basis $\{ \varphi_\lambda \}_{\lambda \in \N}$, as the approximation sets, where the complexity parameter is now the number of nonzero coefficients $r \in \N$. Both $\bar P$ and $\hat P_r$ correspond to retaining the coefficients of largest absolute value, and $\kappa = 1$. Then if $u = \sum_{\lambda} \bu_\lambda \varphi_\lambda$ with $\bu \in \As$, we have
\[
  	\bar{r}(u, \varepsilon) \le \norm{  \bu }_\As^{\frac1s}  \varepsilon^{-\frac1s},
\]
and thus \eqref{rclaim} gives
$  \norm{  u - \hat v   } \leq \varepsilon $ with $\hat v \in \Sigma_{\hat r}$, where 
\[
  \hat r \leq \left( \frac{2+\alpha}{\alpha} \right)^{\frac1s} \norm{  \bu }_\As^{\frac1s}  \varepsilon^{-\frac1s} .
\]
In this manner, this observation appears first in the context of sparse approximations with respect to a given basis, as in adaptive wavelet schemes \cite{Cohen:01,Cohen:02}. 

In the context of hierarchical tensor representations, $r$ is the maximal entry in the tuple of hierarchical ranks. In the simplest case $d=2$, where $\cH = \cH_1 \otimes\cH_2$, we choose $\Scal_r$ as the elements of rank at most $r$, and $\bar P$, $\hat P_r$ are given by the truncated Hilbert-Schmidt expansion. Then if $\sigma(u) \in \As$, in a completely analogous way we obtain $\norm{  u - \hat v   } \leq \varepsilon $ with $\hat v \in \Scal_{\hat r}$ and 
\[
  \hat r = \rank(\hat v) \leq \left( \frac{2+\alpha}{\alpha} \right)^{\frac1s} \norm{  \sigma(u) }_\As^{\frac1s}  \varepsilon^{-\frac1s} .
\]
In this setting, cases where low-rank approximations are particularly attractive are those where $u$ satisfies
\[
    \inf_{\rank(v) \leq r} \norm{ u - v} \leq C e^{-c r^\beta}
\]
 for some $C,c,\beta >0$, corresponding to exponential-type decay of the singular value sequence $\sigma(u)$. In such cases we instead obtain
\[
  \hat r = \rank(\hat v) \leq \bigl(  c^{-1} \ln (C (2+\alpha) \alpha^{-1}\varepsilon^{-1})\bigr)^{1/\beta}  \lesssim \abs{\ln \varepsilon}^{1/\beta}  .
\]
For higher-order hierarchical tensors, the near-optimal projection $\hat{P}_r$ can be done by the routine $\recompress$ implementing the ${H}$SVD hard thresholding operation where,   by Theorem \ref{thm:hier_tensor_recompression_est}, $\kappa = \sqrt{2 d - 3}$. 
Under an analogous assumption 
\[
   \inf_{\max \rr{r}_\mathbb{E}(v) \leq r} \norm{ u - v} \leq C e^{-c r^\beta},
\]
which means that $u \in \AH{(e^{c r^\beta})_{r\in\N_0}}$, one then obtains
\[
  \hat r = \max \rr{r}_\mathbb{E}(\hat v) \leq \Bigl(  c^{-1} \ln \bigl(C (1+(1+\alpha)\kappa) \alpha^{-1}\varepsilon^{-1}\bigr)\Bigr)^{1/\beta}  .
\]
 
In the application of Lemma \ref{lem:coars} to the control of the componentwise sparsity in each tensor mode, as quantified by the contractions $\pi^{(i)}$ in \eqref{eq:contreval}, $r$ is the total number of nonzero entries $\pi^{(i)}_\lambda$ for $i=1,\ldots,d$ and $\lambda \in \mathcal{I}_i$. Proposition \ref{prop:coarsen} provides the corresponding quasi-best approximation procedure $\hat{P}_r$ realized by $\coarsen$. As shown in \cite{BD}, the two reduction procedures for hierarchical tensor decompositions and componentwise sparsity can be combined to achieve the quasi-optimality property \eqref{rclaim}  simultaneously with respect to both types of approximation.

A corresponding adaptive schemes can be realized as in Algorithm \ref{alg:tensor_opeq_solve} using the routines $\coarsen$ and $\recompress$ 
described in Section \ref{ssec:reduction}, combined with residual approximations using routines 
$\rhs(\bbf;\eta)$, $\apply(\bA,\bv;\eta)$. These latter routines generate, for any given target tolerance $\eta >0$,  
error-controlled finite-rank and finitely supported approximations such that
\be 
\label{routines}
 \|\bbf- \rhs(\bbf;\eta)\|_{\ell_2(\cI)} \le \eta,\quad   \|\bA\bv - \apply(\bA,\bv;\eta)\|_{\ell_2(\cI)}\le \eta .
\ee
In all cases, the accuracy requirements are ensured by a posteriori bounds, {\em independently} of any prior assumptions on or 
knowledge of the given finitely parametrized input. 

\begin{algorithm}[!ht]
\caption{$\quad \mathbf{u}_\varepsilon = \solve(\mathbf{A},
\mathbf{f}; \varepsilon)$} \begin{algorithmic}[1]
\Require $\Bigg\{$\begin{minipage}{12cm}$\omega >0$ and $\rho\in(0,1)$ such that
$\norm{\id - \omega\mathbf{A}} \leq \rho$,\\
$c_\bA \geq \norm{\bA^{-1}}$, $\varepsilon_0 \geq c_\bA \norm{\mathbf{f}}$, \\
$ \kappa_1, \kappa_2, \kappa_3 \in (0,1)$ with $\kappa_1 +
\kappa_2 + \kappa_3 \leq 1$, and $\beta_1 \geq 0$, $\beta_2 > 0$.\end{minipage}
\Ensure $\mathbf{u}_\varepsilon$ satisfying $\norm{\mathbf{u}_\varepsilon -
\mathbf{u}}\leq \varepsilon$.
\State $\mathbf{u}_0 := 0$
\State $k:= 0$, $I := \min\{ j \colon \rho^j (1 + (\omega + \beta_1 + \beta_2) j) \leq
\textstyle \frac12 \displaystyle\kappa_1\}$\label{alg:jchoice} 
\While{$2^{-k} \varepsilon_0 > \varepsilon$}
\State $\mathbf{w}_{k,0}:=\mathbf{u}_k$, $j \gets 0$
\While{$j<I$} \label{alg:cddtwo_looptermination_line}
\State $\eta_{k,j} := \rho^{j+1} 2^{-k} \varepsilon_0$
\State $\mathbf{r}_{k,j} := \apply( \mathbf{w}_{k,j} ; \frac{1}{2}\eta_{k,j})
- \rhs(\frac{1}{2}\eta_{k,j})$ 
\State $\mathbf{w}_{k,j+1} := \coarsen\bigl(\recompress(\mathbf{w}_{k,j} - \omega \mathbf{r}_{k,j} ;
\beta_1 \eta_{k,j}); \beta_2 \eta_{k,j} \bigr)$ \label{alg:tensor_solve_innerrecomp}
\State $j\gets j+1$.
\EndWhile
\State $\mathbf{u}_{k+1} := \coarsen\bigl(\recompress(\mathbf{w}_{k,j};
\kappa_2 2^{-(k+1)} \varepsilon_0) ; \kappa_3 2^{-(k+1)}
\varepsilon_0\bigr)$\label{alg:cddtwo_coarsen_line} 
\State $k \gets k+1$
\EndWhile
\State $\mathbf{u}_\varepsilon := \mathbf{u}_k$ 
\end{algorithmic}
\label{alg:tensor_opeq_solve}
\end{algorithm}
\newcommand{\constsvd}{\sqrt{2d-3}}
\newcommand{\constcrs}{\sqrt{d}}

One step in Algorithm \ref{alg:tensor_opeq_solve} takes the basic form
\be
\label{typical}
\bar\bu \leftarrow \coarsen\Big(\recompress\big(\bar\bu -\omega \br_{\eta_A,\eta_f}(\bar\bu)\big);\eta_{rc}\big);\eta_{coa}\Big),
\ee
where $\br_{\eta_A,\eta_f}(\bar\bu) := \apply(\bA,\bar u;\eta_A) - \rhs(\bbf;\eta_f)$ with certain tolerances $\eta_A,\eta_f, \eta_{rc},\eta_{coa}$ of comparable size and geometric decay.
The algorithm has an inner loop, where steps \eqref{typical} are applied to achieve an error reduction, corresponding to the mapping $\mathcal{F}_n$ in \eqref{itertemplate}. In the outer loop, this is followed by a complexity reduction by $\recompress$ and $\coarsen$ with sufficiently large tolerances, see step 11. As shown in \cite{BD}, Lemma \ref{lem:coars} can be applied to infer joint quasi-optimality of hierarchical ranks and mode frame supports provided that the parameters $\kappa_1,\kappa_2,\kappa_3$, which control the relative size of error and complexity reduction tolerances, are chosen as
\begin{equation}
\begin{gathered}
\kappa_1 = \Bigl(1 + (1+\alpha)\bigl(\constsvd + \constcrs +
  \sqrt{(2d-3)d}\bigr)\Bigr)^{-1}\,, \\
  \kappa_2 = \constsvd\,(1+\alpha) \kappa_1\,,\qquad 
  \kappa_3 = \constcrs\,(\constsvd + 1)(1+\alpha)\kappa_1 \,.	
\end{gathered}
\end{equation}
\begin{figure}
	\centering\includegraphics[width=11cm]{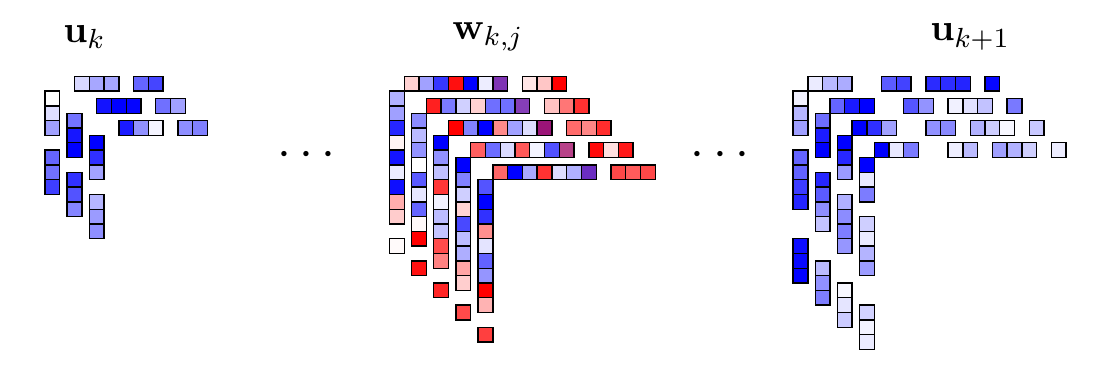}
	\caption{Illustration of Algorithm \ref{alg:tensor_opeq_solve}, for $d=2$ with coefficient arrays determined by sums of outer products of sparse vectors: starting from $\bu_k$, in the inner loop (steps 2--11), ranks are increased and degrees of freedom are activated (shown in red) in $\bw_{k,j}$ for $j=1,\ldots,J$. Small coefficients are removed in the re-approximation step of line \ref{alg:cddtwo_coarsen_line} to obtain $\bu_{k+1}$.}
\end{figure}

\begin{remark}
It is easy to see that for any given tolerance $\varepsilon$, Algorithm \ref{alg:tensor_opeq_solve} terminates after $\mathcal{O}(|\ln \varepsilon|)$ steps,
producing a coefficient array $\bu_\varepsilon$ such that $\|\bu -\bu_\varepsilon\|_{\ell_2(\cI)}\le \varepsilon$.
We emphasize that convergence of the algorithm is guaranteed independently of any additional conditions 
on the unknown solution array $\bu$. The performance of the algorithm, i.e., the numerical cost of computing $\bu_\varepsilon$,
can, of course, only be quantified under certain assumptions on the solution.
\end{remark}

\subsubsection{Complexity bounds in model cases}

Guided by the findings in Section \ref{sec:approximability}, we proceed to formulate the relevant assumptions and {\em benchmark} properties 
for the problem scenarios (I) and (II) for which we will quantify the performance 
of Algorithm \ref{alg:tensor_opeq_solve}. For a detailed account we refer to \cite{BD2,BD3}, in particular regarding the dependence
on $d$ of various constants.

In this regard, similar assumptions are suitable for the model problems in scenario (I) and in scenario (II) with $d<\infty$, since in both cases we expect
\be
\label{bm1}
\bu\in \Acal_\cH(\gamma), \quad \gamma(n)= e^{c n^{\beta}}, \quad\text{for some}\,\,c, \beta>0\,.
\ee
For the sparse approximability in each tensor mode, we may assume
\be
\label{bm2}
\pi^{(i)}(\bu)\in \Acal^s,\quad i=1,\ldots,d,\quad \text{for some $s>0$}.
\ee
Then for any $\varepsilon>0$, the result of Algorithm \ref{alg:tensor_opeq_solve} after the final complexity reduction step satisfies\footnote{If $\beta\leq 1$, the estimate \eqref{eq:complexity_rank} can be quantified more precisely in terms of $\norm{\bu}_{\AH{\gamma}}$; see \cite{BD}.}, as a consequence of Lemma \ref{lem:coars},
\begin{gather} \label{eq:complexity_rank} 
 \max \rr{r}_\mathbb{E}(\bu_\varepsilon)
    \lesssim (\abs{\ln \varepsilon} + \ln d)^{1/\beta} \,,
   \\
 \label{eq:complexity_supp} \sum_{i=1}^d \#\supp(\pi^{(i)}(\bu_\varepsilon)) \lesssim
     d^{1 + s^{-1}} \, \Bigl(\sum_{i=1}^d \norm{ \pi^{(i)}(\bu)}_{\As} \Bigr)^{\frac{1}{s}}
           \varepsilon^{-\frac{1}{s}} \,.
\end{gather}
Note that $\#\supp(\pi^{(i)}(\bu_\varepsilon))$ is the number of degrees of freedom in tensor mode $i$. Thus, in other words, ranks and discretization sizes in the result are {\em quasi-optimal}. Moreover, one has
\begin{gather}
  \label{eq:complexity_ranknorm} 
  \norm{\bu_\varepsilon}_{\AH{\ga_\mathbf{\bu}}}
  \lesssim \sqrt{d}\,
      \norm{\bu}_{\AH{\ga_\mathbf{\bu}}}    \,,   \\
  \label{eq:complexity_sparsitynorm} \sum_{i=1}^d \norm{
  \pi^{(i)}(\bu_\varepsilon)}_{\As} \lesssim d^{1 + \max\{1,s\}} 
      \sum_{i=1}^d \norm{ \pi^{(i)}(\bu)}_{\As}  \,.
\end{gather}
For the proofs of the estimates \eqref{eq:complexity_rank}, \eqref{eq:complexity_supp}, \eqref{eq:complexity_ranknorm}, \eqref{eq:complexity_sparsitynorm}, see \cite{BD,BD2}.
Note that these bounds hold after every execution of the outer loop in Algorithm \ref{alg:tensor_opeq_solve}, with $\varepsilon$ replaced by the accuracy $2^{-k} \varepsilon_0$ of the current iterate $\bu_k$.

On this basis, we obtain estimates for the total computational complexity of the adaptive scheme, which depends mainly on the computational costs of residual approximations in the inner loop. At this point, we need to exploit structural properties of
$\bA$. Specifically, we refer to \cite{BD2,BD3,BCD15} where $s^*$-compressibility of the tensor components of $\bA$  has been established.
The value of $s^*$ depends on diffusion coefficients, the choice of the Riesz-basis, and in scenario (II), on the 
properties of the parametric expansions.
Furthermore, we need to make  some additional technical assumptions
concerning approximability of $\bbf$, including $\bbf \in \AH{\gamma}$. These properties of the problem data can, in principle, be verified.

In case (I), for the bound of the complexity of $\apply(\bA,\bv;\eta)\to \bw_\eta$, we assume in addition that $u\in H^t(\Omega)$ for some $t>1$. One can then bound the ranks $\hat m(\eta,\bv)$ required for the approximate rescaling $\tilde\bS^{-1}_{m(\bv,\eta)}$. Then one obtains estimates of the form 
\be
\label{ranks}
r_i(\bw_\eta)\le (\hat m(\eta,\bv))^2 R r_i(\bv), \quad i=1,\ldots,E,
\ee
where $R$ bounds the representation ranks of $\bT$ (for instance, $R=2$ in the case $A=-\Delta$). Moreover,  whenever $s<s^*$,
\begin{equation}
\label{Acompl}
\begin{gathered}
\#(\supp(\pi^{(i)}(\bw_\eta) ) \lesssim\Big( \sum_{j=1}^d \|\pi^{(i)}(\bv) \|_{\Acal^s}\Big)^{1/s} \eta^{-1/s},\\
 \|\pi^{(i)}(\bw_\eta)\|_{\Acal^s}
 \lesssim d \|\pi^{(i)}(\bv)\|_{\Acal^s},
\end{gathered}
\end{equation}
with hidden constants depending in particular on $\bA$; for details, see \cite[\S6.2]{BD2}. Thus, for a certain approximability range $s<s^*$,
the approximate operator application preserves joint mode frame sparsity.

With \eqref{ranks}, \eqref{Acompl}, and taking into account that the costs of each iteration in the inner loop are dominated by those for the ${H}$SVD of iterates and residual approximations, we arrive at the bound
\begin{equation}\label{totalcomplexity1}
 	\ops(\bu_\varepsilon)\lesssim d^{C_1\ln d} ({1+} \abs{\ln \varepsilon})^{C_2 \ln d + 2 /\beta }\,
       \varepsilon^{-\frac{1}{s}}
\end{equation}
 for the total computational costs of Algorithm \ref{alg:tensor_opeq_solve}, where $C_1,C_2>0$ are constants 
 {\em independent} of $d$ and $\varepsilon$.
Similar (and slightly more favorable) estimates hold for (I) with error control in $L_2$, see \cite{BD3}.

In scenario (II) with $d<\infty$, the situation is similar, but instead of \eqref{ranks} one has the $\eta$-independent bound $r_i(\bw_\eta)\le (d+1) r_i(\bv)$, $i=1,\ldots,E$. As a consequence, the complexity bound improves to
\begin{equation}\label{totalcomplexity2}
 \ops(\bu_\varepsilon)\lesssim d^{C\ln d} \abs{\ln \varepsilon}^{2 / \beta} \varepsilon^{-\frac1s},	
\end{equation}
with $C>0$ independent of $d$ and $\varepsilon$. 

Concerning the dependence on $d$, the bounds \eqref{totalcomplexity1}, \eqref{totalcomplexity2} are in general far from sharp: also in the inner loop of Algorithm \ref{alg:tensor_opeq_solve}, the ranks of iterates typically remain lower than guaranteed by the available upper bounds. However, the estimates \eqref{totalcomplexity1}, \eqref{totalcomplexity2} still show a growth in $d$ that is far below exponential, and thus ensure that low-rank methods can indeed break the curse of dimensionality for these problems also in terms of their total computational cost. 

In scenario (II) with $d=\infty$, we again consider low-rank approximation with a single separation between spatial and parametric degrees of freedom as in \eqref{uelr}. Assuming that the parametric expansion functions $\psi_j$ have multilevel structure according to \eqref{ml},
 one immediately obtains from the results in \cite{BCM:2015} that
$\pi^{(\rs)}(\bu), \pi^{(\rp)}(\bu) \in \Acal^s$ for any $s < \alpha/m$.
In view of the discussion in \ref{sssec:infparam}, for the singular values $\sigma(\bu)$ of $\bu$ this implies $\sigma(\bu) \in \Acal^s$ for any $s < \alpha/m$. Unfortunately, one cannot in general expect stronger summability of $\sigma(\bu)$; that is,  in this setting the singular values of $\bu$ generally decay only \emph{algebraically}.

In this case, as a consequence of Lemma \ref{lem:coars}, one has
\begin{equation}\label{hsrankbound}
		 \rank(\bu_\varepsilon) \lesssim { \varepsilon^{-\frac1{s}}} \norm{\sigma(\bu)}_{\As}^{\frac1{s}}, \quad \norm{\sigma(\bu_\varepsilon)}_{\As} \lesssim \norm{\sigma(\bu)}_{\As},
	\end{equation}
as well as estimates for $\#\supp \pi^{(i)}(\bu_\varepsilon)$, $i=\rs,\rp$, analogous to \eqref{eq:complexity_supp}, \eqref{eq:complexity_sparsitynorm}.

Under certain assumptions on the parametric representation of the diffusion coefficient, the compressibility on $\bA_j$ is quantified in 
\cite{BCD15} based on which    
$\apply(\bv;\eta) \to \bw_\eta$, as in \eqref{eq:bweta}, is shown to 
  exhibit near optimal performance also in this case. Specifically, one has
\[
\begin{gathered}
\rank(\bw_\eta)\lesssim { \eta^{-\frac{1}{s}}}\| \sigma( \bv)\|_{\As}^{\frac{1}{s}} { (1+\abs{\log\eta})^{1/\beta}} 
, \\ 
\| \sigma(\bw_\eta)\|_{\As}\lesssim  \|\sigma(\bv)\|_{\As} { (1+\abs{\log\eta})^{\bar s /\beta }},  
\end{gathered}
\]
as well as
\[
\|\pi^{(\rp)}(\bw_\eta)\|_{\cA^{s}} \lesssim \|\pi^{(\rp)}(\bv)\|_{\cA^{s}} { (1+\abs{\log \eta})^{s / \beta}}
\]
and analogous bounds for $\|\pi^{(\rs)}(\bw_\eta)\|_{\cA^{s}}$. These estimates show 
a logarithmic degradation in the target tolerances. Nevertheless, they allow one to establish the estimate
	\begin{equation}\label{hsopbound}
{\rm ops}(\bu_\varepsilon) \lesssim  {  ( 1 + \abs{\log \varepsilon})^\zeta} \Bigl( { \varepsilon^{-\frac1{s}}} \norm{\sigma(\bu)}_{\As}^{\frac1{s}} \Bigr)^2  \sum_{i\in\{\rs,\rp\}} { \varepsilon^{-\frac1{s}}} \norm{\pi^{(i)}(\bu)}_{\cA^{s}}^{-\frac{1}{s}} ,
	\end{equation}
	where $\zeta>0$ depends on $s$, on $\operatorname{cond}(\bA)$, and on the parameters in Algorithm \ref{alg:tensor_opeq_solve}. Further details can be found in \cite{BCD15}.
Thus, in this case with algebraic decay of singular values, although the arising ranks are quasi-optimal, the quadratic dependence of the cost of the SVD on these ranks leads to substantially less favorable costs of low-rank approximations.  This effect is essentially unavoidable by any low-rank scheme using a separation of spatial and parametric variables.

\begin{remark}
	As outlined in Section \ref{sec:iter}, an alternative to the perturbed Richardson iteration in Algorithm \ref{alg:tensor_opeq_solve} consists in solving a sequence of Galerkin discretizations that are successively refined using residual approximations as considered in Section \ref{sec:residualapprox}. As shown in \cite{AU:18}, such an approach leads to analogous asymptotic complexity bounds, with potential for quantitative improvements in the numerical costs.
\end{remark}

\subsubsection{Soft thresholding}\label{sssec:st}
A second basic approach to control the complexity of iterates in a scheme of the form \eqref{itertemplate} is to choose a reduction operation $\mathcal{R}_n$ which is {\em non-expansive}, that is, $\norm{ \mathcal{R}_n (\bu) - \mathcal{R}_n(\bv) } \leq \norm{ \bu - \bv }$. As a consequence, if $\mathcal{F}_n$ is a contraction, then $\mathcal{R}_n \circ \mathcal{F}_n$ is still a contraction for each $n$. 
This applies, in particular, to the soft thresholding operation, which can be applied entry-wise to sparse expansions, or to the hierarchical singular values of tensors as discussed in Section \ref{ssec:properties}.

Compared to the truncation of the ${H}$SVD by hard thresholding, one has the interesting feature that convergence of the iteration is preserved regardless of the thresholding value. Rather than thresholding with a sufficiently large tolerance whenever a sufficient error reduction has been achieved, one can therefore simply threshold in every iteration step, with parameters that decrease sufficiently slowly. 

Let us take a closer look at this in the case $\mathcal{F}_n = \mathcal{F}$, that is, the underlying iteration is stationary. 
One key observation for establishing quasi-optimality of ranks is that the limit of the iteration with the thresholded mapping $\mfS_\alpha \circ \mathcal{F}$ can be related to the thresholded exact solution.

\begin{lemma}
	Let $\mathcal{F}$ be a contraction with Lipschitz constant $\rho<1$ and unique fixed point $u$, let $\alpha>0$ and let $u^\alpha$ be the unique fixed point of $\mfS_\alpha \circ \mathcal{F}$. Then
	\[
	  (1 + \rho)^{-1} \norm{ \mfS_\alpha( u) - u } \leq \norm{ u^\alpha - u } 
	    \leq ( 1 - \rho)^{-1} \norm{ \mfS_\alpha( u) - u } . 
	\]
\end{lemma}

Up to mildly dimension-dependent constants (recall that $E=2d-3$), the rank reduction by $\mfS_\alpha$ produces quasi-optimal ranks:
\begin{enumerate}[(i)]
\item In the case of algebraic decay of singular values, that is, $\sigma^{(i)}(u) \in \As$ with $s>0$ for $i=1,\ldots, E$, with $C_{u,s} := \max_i \norm{\sigma^{(i)}(u) }_{\As}$, we have $\norm{\mfS_\alpha (u) - u} \leq E C_{u,s}^{1/(2s+1)} \alpha^{2s/(2s+1)}$ with ranks
\[
\max \rr{r}_\mathbb{E}(\mfS_\alpha(u)) \leq  ( C_{u,s} \alpha^{-1} )^{2/(2s + 1)}.
\]
This means that $\norm{\mfS_\alpha (u) - u)} \leq \varepsilon$ with 
\begin{equation}\label{stquasiopt1}
\max \rr{r}_\mathbb{E}(\mfS_\alpha(u)) \leq (E C_{u,s})^{1/s} \varepsilon^{-1/s}.
\end{equation}
\item If $\sigma^{(i)}_k(u) \leq C e^{-c k^{\beta}}$ for $k\in\N$ with $C,c,\beta>0$, then $\norm{\mfS_\alpha (u) - u} \leq C_1 E ( 1 + \abs{\log \alpha})^{1/(2\beta) } \alpha$ with ranks
\[
 \max \rr{r}_\mathbb{E}(\mfS_\alpha(u)) \leq    ( c^{-1} \abs{\log C \alpha^{-1}})^{1/\beta } \leq C_2 ( 1 + \abs{\log \alpha } )^{1/\beta}.
\]
This gives $\norm{\mfS_\alpha (u) - u)} \leq \varepsilon$, for $\varepsilon<1$, with
\begin{equation}\label{stquasiopt2}
  \max \rr{r}_\mathbb{E}(\mfS_\alpha(u)) \leq C_3 (1 + \abs{\log \varepsilon})^{1/\beta},
\end{equation}
where $C_1,C_2,C_3>0$ depend on $C,c,\beta$.
\end{enumerate}

The link from thresholded approximations to the approximability of exact solutions is provided by the following lemma.
\begin{lemma}\label{thresholdinglemma}
Let $u,v\in\Hcal$, $\alpha > 0$, and $\varepsilon >0$ such that $\norm{u-v}\leq \varepsilon$. 
\begin{enumerate}[{\rm(i)}]
\item If $\sigma^{(i)}(u)\in \As$ with $s>0$ for all $i=1,\ldots,E$, then
\[   r_i \bigl(\mfS_\alpha(v)\bigr)  \leq \frac{4 \varepsilon^2}{\alpha^2} + C_s \bigl( \norm{\sigma^{(i)}(u)}_{\As} \alpha^{-1})^{2/(2s+1)} \,.  \]
\item If $\sigma^{(i)}_k(u) \leq C e^{-c k^\beta}$ for $k\in\N$ with $C, c, \beta> 0$, then
\[   r_i\bigl(\mfS_\alpha(v)\bigr)  \leq \frac{4 \varepsilon^2}{\alpha^2} + \bigl( c^{-1} \ln (2C\alpha^{-1}) \bigr)^{1/\beta} \,.  \]
\end{enumerate}
\end{lemma}

To obtain an iterative method with quasi-optimal ranks for all iterates, it thus suffices to ensure that the first terms of order $\varepsilon^2/\alpha^2$ in the above estimates remain comparable to the respective second terms by decreasing $\alpha$ sufficiently slowly.

For elliptic problems in well-conditioned representations $\bA \bu = \bbf$, this can be realized using a Richardson iteration \eqref{Richardson}. 
Let $\omega >0$ be such that $\xi := \norm{ I - \omega \bA} < 1$.
The basic iterative method applied to the present problem has the form
\begin{equation}\label{stiter1}
   \bu^{n+1} = \mfS_{\alpha_n} \bigl( \bu^n - \omega ( \bA \bu^n - \bbf) \bigr), \quad n\geq 0,
\end{equation}
with $\bu^0 =0$ and $\alpha_n \to 0$ determined (according to \cite[Alg.\ 2]{BS:17}) as follows: set $\alpha_0 = \omega\norm{\bbf}_2/(d-1)$, and for a fixed $\bar B > \norm{\bA}$, take
\begin{equation}\label{stiter2}
 \alpha_{n+1} = \begin{cases}
 	  \frac12 \alpha_n, & \text{if } \norm{\bu^{n+1} - \bu^n}_2\leq \frac{1 - \xi }{\xi \bar B} \, \norm{\bA \bu^{n+1} - \bbf},\\
 	  \alpha_n, &\text{else.}
 \end{cases}
\end{equation}

As shown in \cite{BS:17}, the scheme given by \eqref{stiter1}, \eqref{stiter2}, converges linearly and each iterate $\bu^n$ satisfies quasi-optimal rank estimates of the form \eqref{stquasiopt1} or \eqref{stquasiopt2} provided that the exact solution $\bu$ has the corresponding approximability. 

This is potentially stronger than the result for hard thresholding, where one cannot rule out that ranks in the inner loop 
of Algorithm \ref{alg:tensor_opeq_solve} between steps 3 and 11  cumulatively increase  due to repeated (approximate) application of $\bA$
before being reduced to near-optimality in step 11.
In contrast, in \eqref{stiter1}, \eqref{stiter2}, the iterates are returned to quasi-optimality after {\em every single application} of $\bA$. However, in the form given above and in \cite{BS:17}, the soft thresholding method still assumes a {\em fixed discretization} (or an idealized iteration on the full sequence space) and does not yet incorporate adaptive discretizations.

\section{Conclusions and Outlook}

The methods discussed in this article rest on two conceptual pillars. First, they invoke strategies for estimating errors with respect to the underlying continuous problem in highly nonlinear approximations of solutions of high-dimensional PDEs, combining adaptive discretizations with low-rank expansions. Exploiting the mapping properties of the underlying continuous operator is essential. These strategies are based on the approximate {\em evaluation of residuals} in function spaces and are guaranteed to remain computationally feasible even in very high-dimensional settings as they progress from coarse to fine with certified accuracy at each stage.

Second, they employ basic {\em complexity reduction mechanisms} for ensuring quasi-optimality of computed approximations That is, they ensure that their representation complexities remain comparable to those of corresponding best approximations of the same accuracy.
This can be achieved, in particular, by the truncation of the hierarchical singular value decomposition up to a judiciously chosen tolerance. 
An alternate strategy is based on soft thresholding, which has the advantage of preserving the convergence of iterative methods for any thresholding parameter.

The combination of the approximate residual evaluations with such recompression strategies enables the construction of iterative methods that converge to the exact solutions of the continuous problem with near-optimal computational costs. As we have noted, soft thresholding can,
in principle, give slightly stronger bounds on the total computational complexity, since the ranks of all iterates are under control, and the bounds are therefore less dependent on the ranks of operators. This approach, however, has not yet been combined with adaptive discretizations. 

We have confined the discussion to highlighting the essential conceptual mechanisms. Corresponding findings are illustrated  by first numerical experiments in \cite{BD2,BD3,BCD15}; in particular, comparisons with other methods are given in \cite{BD3}. 
While these experiments confirm the near-optimal asymptotic complexity bounds for the resulting methods, much room 
is certainly left  for   optimizing corresponding 
concrete implementations with regard to quantitative practical performance.

Concerning the basic construction of solvers, there is a variety of methods that follow a quite different philosophy in using minimization principles for optimizing tensor decompositions for {\em fixed discretizations}, for instance
 ALS \cite{hrs-als,white-als,ru}, DMRG \cite{vidal,white}, AMEn \cite{DoSa:14}, or Riemannian optimization methods \cite{KSV:2016}. Comparably little is known, however, on their global convergence properties, let  alone the total computational complexity of such methods in relation
 to the output accuracy. As a possible further direction, they could, however, serve as additional inner iterations for accelerating the convergence of error-controlled methods as considered here.
	
	Here we have concentrated on the application to linear operator equations. For other problem classes, such as eigenvalue problems or time-dependent problems, only rather preliminary results on error-controlled low-rank methods are available, see, e.g., \cite{B,AndreevTobler:2015}. Many additional challenges in the application of the basic principles discussed here to such problems  remain open.

\bibliographystyle{amsplain}
\bibliography{Bachmayr_Dahmen_Survey}

\end{document}

%% file: preamble-ams.tex
\usepackage{latexsym,amsmath,amsthm,amsfonts,amssymb,array,graphicx,epsfig,fancyhdr,stmaryrd,algpseudocode,mathtools}
\usepackage[utf8]{inputenc}
\usepackage{dsfont}
\usepackage{color}
\usepackage{graphicx}
\usepackage{mdwlist}
\usepackage{enumerate}
\usepackage{setspace}
\usepackage{multicol}
\usepackage{multirow}
\usepackage{cite}

\usepackage{mathtools}

\usepackage[section]{algorithm}
\usepackage{algorithmicx}
\usepackage{algpseudocode}
\algrenewcommand\algorithmicrequire{\makebox[32pt][l]{\textrm{input}}}
\algrenewcommand\algorithmicensure{\makebox[32pt][l]{\textrm{output}}}
\algrenewcommand\algorithmicfunction{\textrm{function}}
\algrenewcommand\algorithmicwhile{\textrm{while}}
\algrenewcommand\algorithmicdo{}
\algrenewcommand\algorithmicend{\textrm{end}}
\algrenewcommand\algorithmicforall{\textrm{for all}}
\algrenewcommand\algorithmicfor{\textrm{for}}
\algrenewcommand\algorithmicrepeat{\textrm{repeat}}
\algrenewcommand\algorithmicuntil{\textrm{until}}

\DeclareFontFamily{U}{mathx}{\hyphenchar\font45}
\DeclareFontShape{U}{mathx}{m}{n}{
      <5> <6> <7> <8> <9> <10>
      <10.95> <12> <14.4> <17.28> <20.74> <24.88>
      mathx10
      }{}
\DeclareSymbolFont{mathx}{U}{mathx}{m}{n}
\DeclareMathSymbol{\bigtimes}{1}{mathx}{"91}

\theoremstyle{plain}
\newtheorem{theorem}{Theorem}[section]
\newtheorem{lemma}[theorem]{Lemma}
\newtheorem{proposition}[theorem]{Proposition}

\newtheorem{remark}[theorem]{Remark}

\theoremstyle{definition}

\numberwithin{equation}{section}

\newcommand{\N}{\mathds{N}}
\newcommand{\R}{\mathds{R}}
\newcommand{\Z}{\mathds{Z}}

\DeclareMathOperator{\supp}{supp}

\DeclareMathOperator*{\argmin}{arg\,min}

\DeclareMathOperator{\ops}{ops}

\newcommand{\id}{{\rm I}}

\newcommand{\spl}[1]{\ell_{#1}}

\newcommand{\be}{\begin{equation}}
\newcommand{\ee}{\end{equation}}

\newcommand{\Tcal}{{\mathcal{T}}}
\newcommand{\Hcal}{H} 
\newcommand{\Rcal}{{\mathcal{R}}}
\newcommand{\Scal}{{\mathcal{S}}}

\newcommand{\Ical}{{\mathcal{I}}}

\newcommand{\Acal}{{\mathcal{A}}}

\newcommand{\Fcal}{{\mathcal{F}}}

\DeclareMathOperator{\rank}{rank}

\DeclareMathOperator{\apply}{\textsc{apply}}
\DeclareMathOperator{\coarsen}{\textsc{coarsen}}
\DeclareMathOperator{\rhs}{\textsc{rhs}}
\DeclareMathOperator{\recompress}{\textsc{recompress}}
\DeclareMathOperator{\solve}{\textsc{solve}}

\providecommand{\abs}[1]{\lvert#1\rvert}
\providecommand{\bigabs}[1]{\bigl\lvert#1\bigr\rvert}

\providecommand{\norm}[1]{\lVert#1\rVert}

\providecommand{\Bignorm}[1]{\Bigl\lVert#1\Bigr\rVert}

\def\bu{{\bf u}}
\def\bU{{\bf U}}

\newcommand{\bv}{{\bf v}}
\newcommand{\bw}{{\bf w}}

\def\e2{\spl{2}(\nabla^d)}
\newcommand{\cA}{\mathcal{A}}
\def\ga{\gamma}

\newcommand{\As}{{\Acal^s}}

\newcommand{\AH}[1]{{\Acal_{\mathcal{H}}({#1})}}

\newcommand{\rr}[1]{{\mathsf{#1}}}

\newcommand{\Restr}[1]{\operatorname{R}_{#1}}

\newcommand{\beqn}{\begin{equation}}
\newcommand{\eeqn}{\end{equation}}

\newcommand{\bA}{\mathbf{A}}
\newcommand{\bM}{\mathbf{M}}
\newcommand{\bT}{\mathbf{T}}
\newcommand{\bbf}{\mathbf{f}}

\newcommand{\bB}{\mathbf{B}}

\newcommand{\bS}{\mathbf{S}}

\newcommand{\tbS}{{\tilde{\mathbf{S}}}}
\newcommand{\tbT}{{\tilde{\mathbf{T}}}}

\renewcommand{\e}{\varepsilon}
\newcommand{\cU}{{U}}

\newcommand{\br}{\mathbf{r}}
\newcommand{\T}{\mathbb{T}}

\newcommand{\tripnorm}[1]{|\!|\!| #1|\!|\!|}
\newcommand{\HH}{{\mathrm H}}

\newcommand{\bgamma}{\mbox{\textbf{$\gamma$}}}

\newcommand{\uc}{\mathbf{u}}

\newcommand{\mfH}{\mathfrak{H}}
\newcommand{\mfS}{\mathfrak{S}}
\newcommand{\mfR}{\mathfrak{R}}